\documentclass[11pt,onecolumn]{article}
\usepackage[semibold]{ebgaramond}
\usepackage[noend]{algpseudocode}
\usepackage[cmex10]{amsmath}
\usepackage[normalem]{ulem} 
\usepackage[mathscr]{euscript}
\usepackage[utf8]{inputenc}
\usepackage{graphicx}
\usepackage[a4paper, margin=1in]{geometry}
\usepackage{%
	accents,%
	algorithm,%
	amssymb,%
	amsthm,%
	bbm,%
	enumerate,%
	hyperref,%
	mathtools,%
	pgf,%
	pgfplots,%
	subfigure,%
	tikz,%
	tikzscale,%
}

\usepgflibrary{shapes}
\usetikzlibrary{%
	arrows,%
	automata,%
	backgrounds,%
	chains,%
	decorations.pathmorphing,
	decorations.text,%
	matrix,%
	positioning,
	fit,%
	patterns,%
	petri,%
	plotmarks,%
	scopes,%
	shadows,%
	shapes.misc,
	shapes.arrows,%
	shapes.callouts,%
	shapes%
}
\usetikzlibrary{arrows.meta}
\pgfplotsset{compat=1.12}
\pgfplotsset{
	/pgfplots/my legend/.style={
		legend image code/.code={
			\draw[red](-0.05cm,0cm) -- (0.3cm,0cm);%
		}
	}
}

\newlength\tikzheight
\newlength\tikzwidth

\theoremstyle{plain}
\newtheorem{thm}{Theorem}
\newtheorem{prop}[thm]{Proposition}
\newtheorem{cor}[thm]{Corollary}
\newtheorem{lem}[thm]{Lemma}

\theoremstyle{definition}
\newtheorem{defn}[thm]{Definition}
\newtheorem{exmp}[thm]{Example}
\newtheorem{assum}[thm]{Assumption}
\newtheorem{problem}[thm]{Problem}

\theoremstyle{remark}
\newtheorem{rem}[]{Remark}

\newcommand{\E}{\mathbb{E}}

\newcommand{\bI}{\boldsymbol{1}}
\newcommand{\N}{\mathbb{N}}
\newcommand{\R}{\mathbb{R}}
\newcommand{\Z}{\mathbb{Z}}
\renewcommand{\P}{\mathbb{P}}

\newcommand{\bg}{\bar{g}}

\newcommand{\cX}{\mathcal{X}}

\newcommand{\sX}{\mathscr{X}}

\renewcommand{\th}{\tilde{h}}
\newcommand{\tp}{\tilde{p}}
\newcommand{\tpi}{\tilde{\pi}}

\newcommand{\oG}{\overline{G}}
\newcommand{\bp}{\mathbf{p}}

\newcommand{\iid}{\emph{i.i.d.}\ }

\newcommand{\eq}[1]{\begin{align*}#1\end{align*}}
\newcommand{\EQ}[1]{\begin{equation*}#1\end{equation*}}

\newcommand{\EQN}[1]{\begin{equation}#1\end{equation}}

\newcommand{\abs}[1]{\left\lvert#1\right\rvert}

\newcommand{\set}[1]{\left\{#1\right\}}

\newcommand{\SetIn}[1]{\mathbbm{1}_{\set{#1}}}

\renewcommand{\le}{\leqslant}
\renewcommand{\ge}{\geqslant}

\newcommand{\supp}{\mathrm{supp}}

\algnewcommand{\Inputs}[1]{%
	\State \textbf{Inputs:}
	\Statex \hspace*{\algorithmicindent}\parbox[t]{.8\linewidth}{\raggedright #1}
}
\algnewcommand{\Initialize}[1]{%
	\State \textbf{Initialize:}
	\Statex \hspace*{\algorithmicindent}\parbox[t]{.8\linewidth}{\raggedright #1}
}

\newcommand{\remove}[1]{}
\newcommand{\ubar}[1]{\underaccent{\bar}{#1}}

\usepackage{lineno,hyperref}
\modulolinenumbers[5]

\begin{document}

\title{Optimal Pricing in Multi Server Systems}

\author{Ashok Krishnan K. S.~~~~ Chandramani Singh~~~~Siva Theja Maguluri~~~~Parimal Parag\footnote{Authors 1,2 and 4 are with the Indian Institute of Science, Bangalore, India and can be reached at $\{$ashokk,chandra,parimal$\}@$iisc.ac.in. Author 3 is with the Georgia Institute of Technology and can be reached at siva.theja@gatech.edu.}}
\date{}
\maketitle

\begin{abstract} 
	\noindent We study optimal service pricing in server farms where customers arrive according to a renewal process and have independent and identical (\iid) exponential service times and \iid valuations of the service. The service provider charges a time varying service fee aiming at maximizing its revenue rate. The customers that find free servers and service fees lesser than their valuation join for the service else they leave without waiting.  We consider both finite server and infinite server farms. We solve the optimal pricing problems using the framework of Markov decision problems. We show that the optimal prices depend on the number of free servers. 
	We propose algorithms to compute the optimal prices. We also establish several properties of the optimal prices and the corresponding revenue rates in the case of Poisson customer arrivals. We illustrate all our findings via numerical results.  	
\end{abstract}

\section{Introduction}
Server farms refer to centrally maintained collections of computer servers or processors intended to provide a service~(or a class of services) to customers. 
Over the past decade, server farms have mushroomed to keep up with the massive demand for both data storage and computation, which continues to increase at breakneck speed. These include services such as AWS EC2 and Azure \cite{butkiene2020survey}.  
Server farms offer a cost-effective alternative to customers wherein they need not spend initial setup and maintenance of a service facility. 
These also allow customers to dynamically scale resource utilization and provide redundancy against failure of specific hardware. 
However, service providers incur considerable costs on hardware, cooling, power, security etc.  
Sustained proliferation of data farms is contingent on providers profiting  through service charges levied on the customers.     

Optimal service pricing is central to the thriving operation of server farms \cite{xu2013dynamic,chi2015fairness}. 
Service providers' earnings come from service charges levied on the customers. 
Different customers may have different utilities~(or, valuation) of the service. 
Also, in a server farm with a waiting queue, a customer's valuation will also depend on its expected waiting time, i.e., on the queue length on its arrival. 
The customers  opt for the service only if their valuation of the service exceeds service charge. 
Clearly service charges directly impact service provider's revenue. 
These along with customers' valuation also determine servers' occupancy and congestion which in turn governs future customers' valuation. 
We thus see that determining optimal prices is a complex problem. 
The problem is further complicated by the fact that service providers cannot a priori assess customers' valuation though they often know value distributions based on historical data.

We consider a multiple server system that offers service to stochastically arriving customers. Customers' service durations are random. 
We do not assume any waiting queue. 
The service provider sells the service to customers at potentially time varying prices.  
Different customers also have different values of the service. 
The service provider does not know customers' values but knows value distribution. 
A customer who finds at least one idle server on arrival opts for the service if and only if its value exceeds the current service charge. 
The customers who find all the servers busy on arrival leave the system without getting served. 
The service provider aims to maximize the average revenue rate by setting appropriate prices. 
We derive optimal prices as a function of the number of idle servers. 
We also study various properties of the optimal prices and optimal revenue rate vis-a-vis total number of servers, customer arrival rate, average service time etc.

\subsection{Our Contribution}
We assume a service provider with $K$ servers. 
We further assume that the customers arrive according to a renewal process, 
having \iid  inter arrival times, \iid exponential service times 
and \iid values for the service.   We formulate the revenue maximization problem. First, we study the uniform pricing problem as a sub optimal but easy to implement policy, and obtain performance bounds for this policy. Then, we obtain the revenue maximizing pricing policy by solving an associated Markov decision problem. We study the properties of the optimal solution, and compare its performance to that of sub optimal policies discussed previously. 
Following is a preview of our main results.
\begin{enumerate}
	\item We observe that for the system with infinitely many servers~(i.e., $K = \infty$), the optimal service prices are uniform, i.e., 
	independent of the number of occupied servers.  
	\item We study optimal uniform pricing for $K$ server system~($K < \infty$). We derive a bound on the revenue rate for the optimal uniform price. 
	We also study asymptotic revenue rates for uniform pricing  as arrival rates are scaled, and show that limiting revenue can go to zero for certain arrival processes.   
	\item For finite server systems, we frame the revenue rate maximization problem as a continuous time Markov control problem. 
	We show that the optimal prices depend on the number of occupied servers, 
	and can obtained via solving a fixed point iteration. 
	\item We study the dependence of optimal prices and corresponding revenue rates on customer arrival rates, service rates, and the number of servers $K$,  in the case of Poisson customer arrivals. We show that the optimal revenue is increasing in arrival rate, service rate and number of servers.  We also show that the revenue per arrival rate, revenue per service rate and revenue per server are decreasing in their respective variables.  
	\item We illustrate all our findings via numerical results. Our numerical studies also provide additional insights on the behaviour of optimal prices with respect to  arrival and service rates.
\end{enumerate}

\subsection{Related Work}
Cloud computing facilities that host a large number of data servers face the problem of optimizing the utilization of these servers. 
Designing an optimal pricing policy is a crucial step in extracting the best possible revenue from the system~\cite{greenberg2008cost,wu2019cloud}. Since a cloud compute facility can be modelled as a bunch of servers with an associated queueing process, the cloud pricing problem can be studied as a problem of pricing in queues. 
One of the earliest works that studied pricing of queues was~\cite{naor1969regulation}, in which the entry of customers to a queue was regulated using tolls. 
Customers can decide to balk or join the queue, after observing the queue size. Such systems are called \emph{observable}. Customers join the system if the difference between their valuation of the job and the cost of waiting exceeds the admission price to the queue.This translates to a threshold type policy - if the queue length is greater than the threshold, the customers balk; else they join. The optimal threshold may vary, depending on whether we want to maximize the total social utility or the revenue. 
It was shown that in \cite{naor1969regulation} that the socially optimal threshold was higher than the threshold for revenue maximization. 
A subsequent work~\cite{edelson1975congestion} shows that, 
the revenue maximizing and socially optimal toll values can be the same, provided a two-part tariff is imposed. 
There have been a number of other works which looked at extensions of~\cite{naor1969regulation} or at related models. 
The effect of the reward variance on the performance is studied in~\cite{larsen1998investigating}. In~\cite{hassin1986consumer}, the author examines whether it is always optimal for a profit maximizing service provider to hide the queue length from an arriving customer. 
It is shown that there are thresholds of arrival rates, below which it is optimal for the service provider to hide the queue state information, and above which it is optimal to reveal. 
These, and numerous other related works, have been summarized in~\cite{HavivBook,hassin2016rational}.

Optimization of revenue in queueing systems has been extensively studied. In one of the first works in this direction,  \cite{low1974optimal}, the author studies optimal pricing for an $M/M/s$ queue with  finite waiting room. He shows that the optimal prices are monotone increasing in the number of customers waiting in the system. A similar monotonicity result for the price as a function of the number of customers, for a similar system but with no waiting room, is shown in \cite{paschalidis2000congestion}. In~\cite{ChenState}, the authors look at the revenue maximization  problem from the perspective of the service provider. 
They are interested in maximizing the expected discounted revenue, while keeping the queueing model of~\cite{naor1969regulation}. 
They obtain a revenue optimizing threshold queue length beyond which entries are not allowed into the queue. This threshold can be computed numerically. All customers who see a waiting queue length smaller than this threshold, pay a price equal to the difference between their valuation and waiting cost.
In~\cite{borgs2014optimal}, an explicit form is derived for the threshold obtained in the previous work, and they characterize the earning rate asymptotically. However, both aforementioned  works provide explicit solutions in the case of fixed service valuation (or simple valuation distributions, such as a valuation which takes two values). They do not provide explicit solutions for valuations with continuous support and general distributions. In \cite{ziya2006optimal}, the authors study optimal pricing in finite capacity queueing systems. However, they consider the sub optimal class of static prices, where the prices charged by the service provider is independent of the number of customers present in the system. They find the best prices in this class, and study its variation with the number of servers. Another work which looks at optimal pricing in finite capacity queueing system is \cite{maoui2007congestion}. Here, under the assumption that the \emph{generalized hazard rate} of the valuation distribution is strictly increasing, the authors obtain the optimal, revenue maximizing policy. However, this assumption does not hold for all distributions. Another work which looks at dynamic pricing in queues is \cite{feinberg2016optimal}. The authors consider a multi server queueing system with finite waiting room. They prove that an optimal monotone policy exists, under the average reward criterion. Existence of an optimal monotone policy for a system with two tandem queues is provided in \cite{wang2019optimal}. Apart from these, there is substantial literature which looks at revenue optimization of different models of queueing systems using an MDP framework and obtain existence and structural results on the optimal policy. These include works such as \cite{xu2013dynamic,yoon2004optimal}. In \cite{ccil2011dynamic}, the authors study optimal pricing for a two class queueing system, and obtain structural results for the optimal prices. A comprehensive survey of different  dynamic pricing techniques is available in \cite{den2015dynamic}. In a recent work \cite{besbes2019static}, the authors prove the existence a static pricing policy that obtains $78.9\%$ of the optimal profit in a system with multiple reusable resources. They assume that the arrivals form a Poisson process, and further, that  the revenue rate is a concave function of the arrival rate. This static pricing policy is obtained as a function of the optimal (state dependent) pricing policy.

Since explicit computation of optimal prices and revenues is difficult in general, a number of works study the pricing and revenue problem in asymptotic regimes, and obtain useful insights. That dynamic pricing can lead to lower \emph{variability} in the revenue of pricing system, as opposed to static pricing, is shown in \cite{kim2018value}. They use an asymptotic analysis to show that the revenue loss due to randomness is lower for dynamic pricing than static pricing, when the customer valuation is random. An asymptotically optimal pricing is obtained in \cite{ata2013congestion} when the customers are delay sensitive but have fixed valuations, for a system with two classes of customers. An asymptotic approach to the dynamic pricing problem is given in \cite{ccelik2008dynamic}, where the solution to an approximating diffusion control problem is used as a solution. Another asymptotic regime is the large capacity regime, explored in \cite{ata2009near}. They aim to minimize the cost to the customers caused by delay, when the  delay cost is a non-linear function of delay. The authors obtain different optimal policies, corresponding to different types of cost functions in this asymptotic regime.

As opposed to works such as  \cite{xu2013dynamic,low1974optimal,paschalidis2000congestion,feinberg2016optimal,wang2019optimal,yoon2004optimal} which show existence of the optimal policy and proceed to obtain structural insights,in this work, we explicitly obtain the optimal price as a solution of a fixed point equation. Moreover we consider arrival processes with general distribution, which generalizes the Poisson assumption in these works. We do not restrict ourselves to the increasing hazard rate assumption of \cite{maoui2007congestion}, and thus have a more general result. Since we assume valuations with a general distribution, our result is more general than \cite{ChenState}. 

\emph{Notation:} 
Before we proceed, we introduce the following notation that we use throughout in this article.  
We denote the set of positive integers by $\N$, 
the set of non-negative integers by $\Z_+$, 
the set of non-negative reals by $\R_+$, 
the set of first $n$ positive integers by $[n]$, 
and the set of non-negative real vectors of length $n$ by $\R^n_+$. A list of some commonly used symbols in this paper is given in Table \ref{tab:Symbols}, for easy reference.
\begin{center}
	\begin{table}
		\centering
		\begin{tabular}{ |c|c| } 
			\hline
			Symbol & Meaning  \\ 
			\hline
			$\lambda$ & arrival rate\\
			\hline
			$\mu$ & service rate of one server\\
			\hline
			$\rho$ & load factor $\frac{\lambda}{\mu}$\\
			\hline
			$V_i$ & valuation of job $i$ \\
			\hline
			$\oG(u)$ & $\P[V_1\geq u]$ \\
			\hline
			$p_k$ & admission price when $k$ jobs are present in the system \\ 
			\hline
			$\bp$ & price vector $(p_0,...,p_{K-1})$\\
			\hline
			$X(t)$ & number of busy servers at time $t$\\
			\hline
			$\cX$ & $ \set{0, \dots, K}$\\
			\hline
			$\cX^{'}$ & $ \set{0, \dots, K-1}$\\
			\hline
			$R(K,\bp)$ & revenue rate for $K$-server system with price vector $\bp$\\
			\hline
			$\bp^*$ & optimal price vector for $K$ server system\\
			\hline
			$p^*_K$ & optimal uniform price for $K$ server system\\
			\hline
			$\pi$ & marginal distribution of number of busy servers seen by arriving customer\\
			\hline
			$\phi(s)$ & Laplace Stieltjes transform of interarrival time\\
			\hline
			$\theta^*$ & optimal revenue rate\\
			\hline
		\end{tabular}
		\label{tab:Symbols}
		\caption{List of common symbols}
	\end{table}
	
\end{center} 
\section{System Model}
\label{sec:sys-model}

We model a compute cluster of $K$ servers as a queuing system, 
where jobs arrive with some service time and a valuation. 
The price of admission into the compute cluster is updated at each job arrival. 
If the admission price is smaller than the valuation, 
then the job is admitted into the system. 
The job pays the admission price to the compute cluster. 
In this case, the compute cluster earns the revenue equal to the admission price, 
and the job leaves upon service completion. 
If admission price is larger than the valuation, 
the job leaves and never returns.  A $5$ server system is depicted in fig. \ref{fig:SysModel}. Two servers are occupied, and a new arrival with valuation $V$ attempts to join the system.

\begin{figure}
	\centering
	\begin{tikzpicture}[scale=0.7]
	\draw[blue] (0,0)--(5,0)--(5,1)--(0,1)--(0,0);
	\draw[blue] (1,0)--(1,1);
	\draw[blue] (2,0)--(2,1);
	\draw[blue] (3,0)--(3,1);
	\draw[blue] (4,0)--(4,1);
	\draw[-{Latex[length=3mm]}] (2.5,4.5)--(2.5,3.5);
	\node (n1) at (2.5,3) {price $p_2$};
	\draw[-{Latex[length=3mm]}] (2.5,2.5)--(2.5,1.3);
	\draw[-{Latex[length=3mm]}] (3.5,3)--(5.5,3);
	\node (n2) at (3.5,2) {$V\ge p_2$};
	\node (n3) at (4.25,3.4) {$V< p_2$};
	\draw[blue,fill=white] (2.5,5) circle (10pt);
	\draw[blue,fill=white] (0.5,0.5) circle (10pt);
	\draw[blue,fill=white] (1.5,0.5) circle (10pt);
	\node (n4) at (4.5,5) {job, $V\sim G$};
	\node (n5) at (6,0.2) {$\exp(\mu)$};
	\node (n6) at (6,0.7) {servers};
\end{tikzpicture}
	\caption{
		We depict a 5 server system. 
		A job with valuation $V$ arrives when two servers are busy. 
		The admission price is $p_2$ and the job joins the system if its valuation $V \ge p_2$.
	}
	\label{fig:SysModel}
\end{figure}
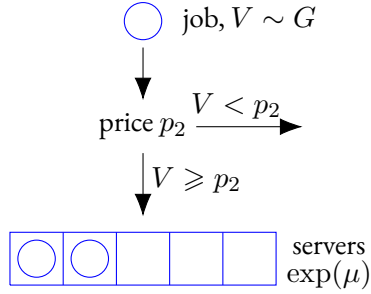

The arrival process is modeled as a renewal process with \iid inter-arrival times having mean  $\frac{1}{\lambda}$. 
Arrival processes are typically modeled by Poisson processes in the literature~\cite{hassin2016rational}. 
Our model is a generalization of this assumption, 
where the sequence of interarrival times $U \triangleq (U_n \in \R_+: n \in \N)$ remains \iid however with a general distribution $F: \R_+ \to [0,1]$. 
The sequence of arrival instants of customers is denoted by $A \triangleq (A_n \in \R_+: n \in \N)$, such that renewal instants $A_n = \sum_{i=1}^nU_i$. 
We denote the counting process associated with the arrival sequence by $N: \R_+ \to \Z_+$ such that 
\EQ{
	N_t \triangleq \sum_{n \in \N}\SetIn{A_n \le t}
} 
is the number of arrivals until time $t$. 

Service time requirements of arriving jobs at compute clusters can be modeled as \iid random variables with a shifted exponential distribution~\cite{bitar2017minimizing,badita2020optimal}, 
with a constant start-up time and a random memoryless service time. 
When the job sizes are large\footnote{When the job sizes are large, the mean of the memoryless service time dominates the constant start-up time.}, exponential distribution is a good approximation for the job service requirement. 
As such, we assume that the service time requirements of arriving jobs is an \iid random sequence $S \triangleq (S_n \in \R_+: n \in \N)$, distributed exponentially with mean $\frac{1}{\mu}$. 

A natural assumption would be to assume that service time requirements affect the job valuation, i.e. higher the service time requirement, larger the valuation. 
However, this assumption has two caveats, the first that the job is aware of its  requirements \emph{apriori}, 
and the second that all jobs are valued in a homogeneous manner. 
In practice, jobs maybe unaware of service time requirements, and they maybe valued heterogeneously.
To keep our model general and analytical tractable, 
we assume that each job has a random \iid positive valuation sampled from a continuous  distribution $G: \R_+ \to [0,1]$. 
We denote the \iid random sequence of job valuations by $V \triangleq (V_n \in \R_+: n \in \N)$ with finite mean $\E V_1$.   
We note that this remains a more general assumption, when compared to constant valuation considered in the literature~\cite{ChenState}. Random valuation models the scenario where the customers are not identical in their assessment of the value of the job. However, they are drawn from a homogeneous population. We assume that the distribution $G$ is known. However, in general it may be necessary to estimate this distribution. For example, see \cite{cong2018developing} where the authors use kernel density estimation methods to estimate $G$.

Recall that we have a finite compute cluster with $K$ servers, 
and we assume that incoming jobs join a unique\footnote{We are not considering redundant replication of jobs, which is an interesting future direction. We will see that our problem remains difficult even without redundancy.} idle\footnote{This model can be extended to the case when jobs join the queue if all $K$ servers are busy. In this case, the price will depend on the number of people existing in the queue, and the state space of possible prices increases.} server if admitted.  
That is, a job leaves if either its valuation is lower than the admission price or all $K$ servers are busy. 
We assume that the server sets a price, that depends only on the number of busy servers at any job arrival instant.  
That is, if we let $k$ be the number of busy servers at a job arrival, then the admission price is $p_k$.  The number of busy servers represents the resource crunch at the service provider. It is  reasonable to expect the service provider to set its prices as a function of this number.
To capture the effect of a job leaving when all $K$ servers are busy, 
we can define the price $p_K \triangleq \infty$. 
Therefore, if there are $k$ busy servers at arrival instant of $n$th job with valuation $V_n$, 
then we can indicate its admission by $\SetIn{V_n \ge p_k}$, 
and the revenue earned by the cluster by $p_k\SetIn{V_n \ge p_k}$.  Note that in our model a customer leaves when no free server is available, or when the price posted is large. Such a  model  is common in the literature and is referred to as a \emph{loss model}  \cite{xu2013dynamic,paschalidis2000congestion,besbes2019static,ellens2012performance}. This corresponds to a situation where the service provider is not a monopoly - there are other service providers to whom the customer can turn to, when the server under consideration is busy or expensive.

We denote the number of busy servers in the system at time $t$ by $X(t) \in  \cX \triangleq \set{0, \dots, K}$. 
Since the admission price depends only on the number of busy servers at the arrival instants, 
it follows from the memoryless property of service times that the number of busy servers specify the system state completely. 
Since we have set $p_K = \infty$, the state space $\cX$ can be reduced to \ $\cX^\prime \triangleq \set{0, \dots, K-1}$. 
We denote a state-dependent price vector by $\bp = (p_0,\dots,p_{K-1}) \in \R_+^{\cX^\prime}$. 
We denote the number of busy servers in the system seen by $n$th arriving customer as $Z_n \triangleq X(A_n^-)$. 
We denote the revenue earned by the cluster until time by $R(t)$, 
which can be written as
\EQN{
	\label{eqn:CumRev}
	R(t) = \sum_{n = 1}^{N_t}\sum_{k=0}^{K-1}p_k\SetIn{V_n \ge p_k}\SetIn{Z_n=k}.
}
The limiting revenue rate for this $K$ server system with the state-dependent price vector $\bp$  is denoted by 
\EQN{
	\label{eqn:LimRevRate}
	R(K,\bp) \triangleq \lim_{t\to\infty}\frac{\E R(t)}{t}.
}
Our main goal is to find the state-dependent pricing vector $\bp$ that maximizes revenue. 
Formally, we solve the following problem. 
\begin{problem} 
	\label{prob:OptPriceVector}
	Find the optimal price vector $\bp^\ast \in \R_+^{\sX^\prime}$ that maximizes the limiting system revenue rate $R(K, \bp)$. 
	That is, we wish to find 
	\EQ{
		\bp^\ast \triangleq \arg\max\set{R(K,\bp):\bp \in \R_+^{\sX^\prime}}.
	}
\end{problem}
Denoting a vector of all ones by $\bI \in \R_+^{\cX'}$ and a fixed price $p \ge 0$, 
we can denote the \emph{uniform price} vector by $p\bI$. 
In this case, the price charged to a customer is independent of the state of the system. 
We next find the uniform price that maximizes the revenue rate. 
\begin{problem} 
	\label{prob:OptPriceUnif}
	Find the uniform price $p$ that maximizes the limiting system revenue rate $R(K, p\bI)$. 
	That is, we wish to find 
	\EQ{
		p^\ast \triangleq \arg\max\set{R(K,p\bI): p \in \R_+}.
	}
\end{problem}
In most systems, calculating the optimal uniform price turns out to be much simpler  than obtaining the optimal price vector $\bp*$. This also provides a benchmark for comparing the optimal policy and quantifying the improvement.
We denote the optimal revenue rate by $R^\ast=R(K,\bp^\ast)$, and compare it to the revenue rate $R(K,p^\ast\bI)$ for the best uniform pricing.
\begin{rem}
	In this paper, we assume that the price charged does not depend on the service time. In contrast, in cloud computing systems such as Amazon EC2 and Microsoft Azure, the customers are charged based on their service time. However, the results in this paper are also applicable in such settings with $p_k$ being interpreted as price per unit service. This can be understood as follows. Suppose $S_i$ is the random service duration of the $i$-th job, then its price is $p_kS_i$, and its expected value is $\frac{p_k}{\mu}$. So, the mean revenue expression, the Bellman's equation characterizing the optimal pricing etc. remain unchanged the same except
	for a constant scaling factor $\frac{1}{\mu}$. Consequently, the optimal pricing analysis and and the properties of the optimal prices also continue to hold.  
\end{rem}

\section{Computation of Revenue Rate}
\label{section:ComputeRevRate} 
Recall that the $n$th customer sees $Z_n = X(A_n^-)$ busy servers in the system. 
We denote the indicator to the event that the job valuation of $n$th customer is higher than the system admission price,  
by $e_n \triangleq \SetIn{V_n \ge p_{Z_n}}$.  
From the memoryless property of service time requirements, 
state dependent admission pricing, 
and the \iid nature of job valuations, 
it follows that the process $((Z_n,e_n) \in \cX \times \set{0,1}: n \in \N)$ evolves as a discrete time Markov chain with finite state space. 

We define $i^\ast \triangleq \min\set{i \in \cX: p_i > \supp(G) \text{ or }p_i = \infty}$. 
Since the valuations are \emph{i.i.d.},  
it can be verified that this Markov chain is irreducible and aperiodic over the reduced state space $\set{0, \dots, i^\ast}\times\set{0,1}$. 
It follows that this reduced Markov chain has a positive invariant distribution $\tpi$. 
For ease of notation, we can extend this distribution $\tpi$ to the entire state space $\cX\times \set{0,1}$ 
by defining $\tilde{\pi}(k,u) = 0$ for all $k > i^\ast$ and $u \in \set{0,1}$. 
Since valuations are \emph{i.i.d.}, 
conditioned on the number of busy servers $Z_n$ seen by the incoming arrival, 
the conditional mean of the random variable $e_n \in \set{0,1}$ is $\E[e_n| Z_n] = \oG(p_{Z_n})$.   
That is, $\oG(p_k)$ is the admission probability of an incoming customer that sees $k$ busy servers. 
Let $\pi \triangleq (\pi_k: k \in \cX)$ be the marginal distribution of the number of busy servers seen by an incoming customer. 
In terms of the marginal distribution $\pi$ and admission probability $\oG(p_k)$, 
we can write the joint distribution $\tpi$ as 
\begin{xalignat}{2}
	\label{eqn:JointDist}
	&\tpi(k,1) = \oG(p_k)\pi_k,&
	&\tpi(k,0) = G(p_k)\pi_k.
\end{xalignat} 

\begin{thm}
	\label{thm:MeanRevenue}
	Given the marginal distribution $\pi$ and the state-dependent arrival rate $\lambda_k \triangleq \lambda\oG(p_k)$, 
	the limiting mean revenue rate for the cluster with state-dependent price vectors $\bp$ is 
	\begin{equation}
	\label{eqn:MeanRevenue}
	R(K,\bp) =\sum_{k=0}^{K-1}\pi_k\lambda_kp_k.
	\end{equation}
\end{thm}
\begin{proof}
	From Eq.~\eqref{eqn:CumRev} for the cumulative revenue $R(t)$ until time $t$, 
	we observe that the revenue earned by the cluster for $n$th arriving customer is denoted by 
	$R(Z_n, e_n) = p_{Z_n}e_n$. 
	Since $N_t$ is a counting process for the arrival renewal process, we have  $\lim_{t\to\infty}\frac{N_t}{t}= \lambda$  almost surely. 
	Hence, we can write,
	\begin{equation*}
	\lim_{t\to\infty}\frac{R(t)}{t} = \lambda\lim_{t\to\infty}\frac{1}{N_t} \sum_{n = 1}^{N_t}R(Z_n,e_n).
	\end{equation*}
	By an ergodic theorem for Markov chains (Theorem 1.10. of~\cite{Norris1997}), it follows that, almost surely,
	\begin{equation*}
	\lambda\lim_{N_t\to\infty}\frac{1}{N_t}\sum_{n=1}^{N_t}R(Z_n, e_n) = \lambda\sum_{k=0}^{K-1}p_k\sum_{u \in \set{0,1}}u\tpi(k,u).
	\end{equation*}
	From Eq.~\eqref{eqn:JointDist} for $\tpi(k,1)$ and the definition of state-dependent arrival rate $\lambda_k$, we see that, almost surely,
	\begin{equation*}
	\lim_{t\to\infty}\frac{R(t)}{t} =\sum_{k=0}^{K-1}\pi_k\lambda_kp_k.
	\end{equation*}
	Since the revenue rate is upper bounded by average valuation of all incoming customers, 
	we get $R(t)/t \le (\sum_{n=1}^{N_t}V_n)/t$. 
	Since the valuation sequence $V$ is independent of the interarrival sequence $U$, 
	it follows from the strong law of large numbers~\cite[Theorem 5.4.2]{chung2001course} that the upper bound converges to $\lambda\E V_1$ almost surely. 
	From the renewal reward theorem~\cite[Theorem 3.6.1]{ross1996stochastic}, 
	we see that $\lim_{t\to \infty}(\sum_{n=1}^{N_t}V_n)/t = \lambda\E V_1$. 
	It follows from~\cite[Theorem 4.5.4]{chung2001course}, 
	that $({R(t)}/{t}: t > 0)$ is a uniformly integrable family of random variables. 
	Consequently, we have  
	\begin{equation*}
	\lim_{t\to\infty}\frac{\E R(t)}{t}=\E \lim_{t\to\infty}\frac{R(t)}{t} =\sum_{k=0}^{K-1}\pi_k\lambda_kp_k.
	\end{equation*}
\end{proof}
We will assume that  the optimal price vector defined in Problem~\ref{prob:OptPriceVector} exists and is finite.

\begin{assum}
	\label{assum:UniqueMax}
	There exists a finite optimal price $\bp^*$ such that,
	\EQ{
		R(K,\bp^*)=\max_{\bp\in \R_+^{\sX^\prime}}R(K,\bp).
	}
\end{assum}
We are interested in finding this $\bp^*$ whenever it exists.

\begin{rem}
	Consider the discrete-time discrete-state process $Z \triangleq (Z_n \in \cX: n \in \N)$, 
	that denotes the number of busy servers seen by an incoming arrival. 
	In the $n$th interarrival time $U_n$, 
	the number of departures from $Z_n=k$ busy servers is denoted by random variable $N_k(U_n)$. 
	Conditioned on duration $U_n$ and $Z_n = k$, 
	the probability of $i$ departures is given by 
	\EQ{
		P\set{N_k(U_n) = i} = \binom{k}{i}(1-e^{\mu U_n})^ie^{-(k-i)\mu U_n},
	}
	for $i \in \set{0, \dots, k}$. Since the interarrival time sequence $U$ is \iid  with general distribution $F$, 
	we can write the probability of $0 \le i \le k$ departures from $k$ busy servers, as 
	\begin{equation}
	\label{eqn:NumDeparturesSampledProcess}
	\alpha_{k,i} \triangleq \E P\set{N_k(U_n) = i} = \int dF(x)P\set{N_k(x) = i}.  
	\end{equation}
	Then, we can write the homogeneous probability for the Markov chain $Z$ to transition from state $k \in \cX$ to state $j\in\set{0, \dots, \min\set{k+1,K}}$ as
	\begin{equation}
	\label{eqn:TransitionSampledProcess}
	\oG(p_k) \alpha_{k+1,k+1-j} + G(p_k)\alpha_{k,k-j}. 
	\end{equation}
	Therefore,  one can find the transition probability matrix for the sampled Markov chain $Z$,  
	for any general interarrival distribution $F$. 
	It follows that the limiting distribution of the number of busy servers can be evaluated at least numerically. 
\end{rem}

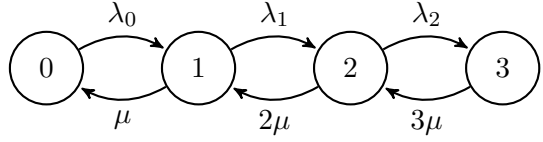
\begin{figure}[hhhh]
	\centering
	\begin{tikzpicture}
	[->,>=stealth',shorten >=1pt,auto,node distance=2.0cm,thick]
	\tikzstyle{every state}=[fill=none,draw=black,text=black]			
	\node[state] (0) at (0,0) {$0$};
	\node[state] (1) [right of=0] {$1$};
	\node[state] (2) [right of=1] {$2$};
	\node[state] (3) [right of=2] {$3$};
	
	\path (0) edge [bend left] 		node {$\lambda_0$} (1) 
	(1) edge [bend left] 		node {$\lambda_1$} (2)
	edge [bend left]		node {$\mu$}		(0)
	(2) edge [bend left] 		node {$\lambda_2$} (3)
	edge [bend left]		node {$2\mu$}		(1)
	(3)	      edge [bend left]		node {$3\mu$}		(2);
\end{tikzpicture}
	\caption{Number of busy servers for state dependent price, where we have three identical servers.}
	\label{Fig:StateDependentPrice}
\end{figure}
\begin{rem}[Kelly~\cite{Kelly2011}]
	\label{thm:EqDist}
	The computation of marginal distribution $\pi$ of the number of busy servers, is straightforward for Poisson arrivals. 
	In this case, the evolution of the number of busy servers forms a birth-death Markov process, 
	with transitions depicted in Fig.~\ref{Fig:StateDependentPrice}.
	Due to PASTA property, the distribution of number of busy servers seen by incoming customers is identical to the stationary distribution $\pi$ of this Markov process. 
	In particular, the distribution $\pi$ is given in terms of the load factor $\rho\triangleq\frac{\lambda}{\mu}$ as 
	\EQN{
		\label{eqn:EqDistK}
		\pi_k = 
		\begin{cases}
			\pi_0\frac{\rho^k}{k!}\prod_{j=0}^{k-1}\oG(p_j),& k \neq 0,\\
			\left[1+\sum_{k=1}^{K}\frac{\rho^k}{k!}\prod_{j=0}^{k-1}\oG(p_j)\right]^{-1}, & k =0.
		\end{cases}
	}
\end{rem}

We showed in Theorem~\ref{thm:MeanRevenue}, 
that the limiting revenue rate $R(K,\bp)$ can be written as a function of state-dependent price vector $\bp$, 
marginal distribution $\pi$, 
and state-dependent arrival rates $\lambda_k$. 
Hence, the optimal price vector depends on the marginal distribution $\pi$. 
This marginal distribution is not easy to compute for the case of general inter arrival distribution, 
and its properties are not easy to establish even when inter arrival times are exponential. 
Therefore, we first consider a simple sub-class of prices, the uniform prices, where the price is independent of the state.

\section{Uniform Pricing}
\label{section:UniformPricing}
In this section, we will consider uniform pricing, not only when the number of servers is finite, but also when it is countably infinite. 
We show that uniform pricing is optimal in the infinite server scenario. 
Hence, optimizing the revenue over the simpler class of uniform prices is a reasonable solution, 
when the number of servers is large. 

From Theorem~\ref{thm:MeanRevenue},  the following corollary is immediate for the revenue rate under uniform pricing.   
\begin{cor} 
	The mean revenue rate for $K$-server system under uniform pricing $\bp = p\bI$ is 
	\EQN{
		\label{eqn:RevenueUniformPriceFiniteServers}
		R(K,  p\bI) = \lambda p\oG(p)(1-\pi_K(p)).
	}
\end{cor}
The revenue rate depends on the probability $1-\pi_K(p)$ of arriving jobs seeing at least one idle server. 
This probability depends on the uniform price $p$. 
Hence, we  obtain an expression for the blocking probability $\pi_K(p)$
to understand the revenue rate dependence on the uniform price $p$. 
\begin{prop}
	\label{prop:Blocking}
	Consider a $K$-server system with uniform price $\bp = p\bI$. 
	We denote the Laplace Stieltjes transform (LST)  of the interarrival time  $U$ by $\phi: \R \to \R_+$, 
	which is defined by $\phi(s) \triangleq \E[e^{-sU}]$ for all $s \in \R$.
	Defining 
	\EQN{
		\label{eqn:beta}
		\beta_j \triangleq \prod_{m=1}^j\frac{1-\phi(m\mu)}{\phi(m\mu)}, 
	} 
	we can write the limiting probability of finding all $K$ servers busy as
	\EQN{
		\label{eqn:BlockProbUnifPrice}
		\pi_K(p)=(\sum_{j=0}^K \binom{K}{j}\oG(p)^{-j}\beta_j)^{-1}.
	}
\end{prop}
\begin{proof}
	Recall that interarrival times $U$ for jobs are \iid with common distribution $F$. 
	For uniform pricing $\bp = p\bI$, the admission indicator sequence $e \triangleq (e_n: n \in \N)$ are \iid Bernoulli with $\E e_n = \oG(p)$.  
	We write the number of arrivals between $(n-1)$th and $n$th admission as $T_n$, 
	and observe that $T \triangleq (T_n: n \in \N)$ is an \iid geometric sequence with success probability $\oG(p)$, 
	and independent of inter-arrival sequence.   
	We denote the inter-arrival times for admitted job as $\tilde{U} \triangleq (\tilde{U}_n: n\in \N)$, 
	where $\tilde{U}_{n} \triangleq \sum_{k=1}^{T_n}U_k$. 
	It follows that $\tilde{U}$ is \iid and thinned version of the original arrival process $U$. 
	We can write the LST for the inter-arrival times of admitted jobs in terms of thinning probability $\oG(p)$ as 
	\EQN{
		\label{eqn:LSTAdmittedJobInterArrival}
		\tilde\phi(x)=\sum_{n=1}^{\infty}{\phi(x)}^n{G(p)}^{n-1}\oG(p)=\frac{\oG(p)\phi(x)}{1-G(p)\phi(x)}.
	}	
	We observe that the evolution of the $K$-server pricing system under uniform pricing, 
	is identical to that of a $G/M/K/K$ queueing system with \iid inter-arrival times $\tilde{U}$ and $K$ \iid servers with exponential service rates $\mu$. 
	Therefore, the limiting blocking probability for this stable $G/M/K/K$ system can be written, 
	using the Palm's formula~\cite{takacs1957probability}, as 
	\begin{equation*}
	\pi_K(p) =\frac{1}{\sum_{j=0}^K \binom{K}{j}\prod_{m=1}^j\frac{1-\tilde\phi(m\mu)}{\tilde\phi(m\mu)}}.
	\end{equation*}
	Result follows from Eq.~~\eqref{eqn:LSTAdmittedJobInterArrival}, which implies that $\frac{1-\tilde\phi(m\mu)}{\tilde\phi(m\mu)} = \frac{1}{\oG(p)}\big(\frac{1-\phi(m\mu)}{\phi(m\mu)}\big)$. 
	
\end{proof}
From above proposition, we can make the following observations for the limiting blocking probability. 
\begin{prop}\label{pikayincrease}
	For the finite server system under uniform pricing, the limiting blocking probability is nonincreasing in  
	\begin{enumerate}[(a)]
		\item 
		uniform price for a fixed number of servers,
		\item 
		number of servers for a fixed uniform price. 
	\end{enumerate}
\end{prop}
\begin{proof}
	We recall the form of blocking probability $\pi_K(p)$ given in Eq.~\eqref{eqn:BlockProbUnifPrice} for $K$-server system under uniform price $p$. 
	\begin{enumerate}[(a)]
		\item 
		Blocking probability $\pi_K(p)$ is non-decreasing in $\oG(p)$, and the tail probability $\oG(p)$ is non-increasing in uniform price $p$. 
		\item  From the definition of $\beta_j = \prod_{m=1}^j\frac{1-\phi(m\mu)}{\phi(m\mu)}$ in Eq.~\eqref{eqn:beta}, 
		the binomial identity $\binom{K+1}{j} = \binom{K}{j} + \binom{K}{j-1}$, 
		and positivity of all terms, we observe that 
		\EQ{
			\pi_{K+1}(p)^{-1} 
			\ge \pi_K(p)^{-1}.
		}
		
	\end{enumerate}
\end{proof}
\begin{rem}
	Above proposition implies that a higher price leads to a lower blocking probability for the same number of servers, 
	since some jobs will leave without joining. 
	It also implies that block probability is reduced by increasing the number of servers while keeping the price fixed. 
\end{rem}
\begin{defn}
	For a $K$-server system, we can define the optimal uniform price $p^{\ast}_K$ as the price that maximizes the mean revenue rate under uniform pricing. 
	That is, 
	\EQ{
		p^{\ast}_K \triangleq \arg\max_{p > 0}R(K, p\bI) =  \arg\max_{p > 0} \lambda p \oG(p)(1 - \pi_K(p)).
	}
	The corresponding revenue rate for this price is $R(K,p^*_K\bI)$.  
\end{defn}
\subsection{Properties of revenue rate under uniform pricing}
\label{subsec:PropRevRate}

We now show that this optimal revenue increases with the number of servers. 
\begin{lem}
	\label{lem:RevenueMonotoneUniformPricing}
	The mean revenue rate for a finite server system under uniform pricing is increasing in the number of servers. 
\end{lem}
\begin{proof}
	Consider the optimal uniform price $p^\ast_K$ for $K$ server system. 
	When this uniform price is applied to a $K+1$ server system, 
	then the mean revenue rate of this system is given by Eq.~\eqref{eqn:RevenueUniformPriceFiniteServers}, as 
	\begin{equation*}
	R(K+1, p^\ast_K\bI)=\lambda p^*_K\oG(p^*_K)(1-\pi_{K+1}(p^*_K)).
	\end{equation*}
	From the monontonicity of blocking probability with the number of servers in Lemma~\ref{pikayincrease}, for finite server system under uniform price, it follows that $\pi_{K+1}(p)\le \pi_K(p)$. 
	Therefore, we have 
	\begin{equation*}
	R(K,p^*_K\bI) 
	\le \lambda p^*_K\oG(p^*_K)(1-\pi_{K+1}(p^*_K))= R(K+1,p^\ast_K\bI).
	\end{equation*}
	Since the optimal uniform price for $K+1$ server system is $p^\ast_{K+1}$, 
	we obtain that $R(K+1, p^\ast_K\bI) \le R(K+1, p^\ast_{K+1}\bI)$ and the result follows. 
\end{proof}
\begin{rem}
	We consider the uniform pricing for the limiting case when the number of servers grow unboundedly large. 
	If the uniform price is $p$, 
	then any arriving job with valuation higher than $p$ joins the system. 
	Since there is no blocking due to unavailability of servers, the mean revenue rate for the limiting system is $\lambda p\oG(p)$. 
	Therefore, the optimal uniform price for infinite server system is given by
	\begin{equation}
	\label{eqn:RevenueUniformPriceInfiniteServers}
	p^\ast_{\infty}\triangleq \arg\max_p p\oG(p).
	\end{equation}
\end{rem}
We next see that the optimal uniform price for infinite server system is lower than the optimal uniform price for any finite server system.
\begin{lem} 
	Let $p^{\ast}_\infty$ defined in Eq.~\eqref{eqn:RevenueUniformPriceInfiniteServers} and $p^{\ast}_K$ defined in Eq.~\eqref{eqn:RevenueUniformPriceFiniteServers} be the optimal uniform prices for infinite and finite $K$-server systems respectively.  
	Then, $p^{\ast}_K \ge p^{\ast}_\infty$ for all finite $K$.
\end{lem}
\begin{proof} 
	Let $\pi_K(p^{\ast}_K)$ and $\pi_K(p^{\ast}_\infty)$ be the blocking probabilities for $K$-server system with uniform prices $p^{\ast}_K\bI$ and $p^{\ast}_\infty\bI$ respectively. 
	From the definition of optimal uniform price for infinite server system, 
	it follows that $p^\ast_\infty\oG(p^\ast_\infty) \ge p^\ast_K\oG(p^\ast_K)$. 
	From the definition of optimal uniform price for finite server systems, 
	it follows that 
	\begin{align*}
	(1-\pi_K(p^\ast_K))p^\ast_\infty\oG(p^{\ast}_\infty) 
	&\ge (1-\pi_K(p^{\ast}_K))p^{\ast}_{K}\oG(p^{\ast}_{K}) \\
	&\ge (1-\pi_K(p^{\ast}_{\infty}))p^{\ast}_\infty\oG(p^{\ast}_\infty).
	\end{align*}
	Therefore, we have $\pi_K(p^{\ast}_\infty) \ge \pi_K(p_{K}^{\ast})$. 
	The result follows from the monotone decrease of blocking probability $\pi_K$ in uniform price $p$ From Lemma~\ref{pikayincrease}. 
\end{proof}
We now establish that the mean revenue rate in the infinite server system is maximized by the optimal uniform pricing.
\begin{prop}
	The optimal uniform pricing $p^*_{\infty}$ maximizes the mean revenue rate for infinite server system.
\end{prop}
\begin{proof}
	By definition of the optimal revenue rate for $K$ servers, 
	the optimal revenue rate $R(K, \bp^\ast)$ with state dependent dependent pricing $\bp^\ast$ is greater than the maximum revenue rate $R(K,p_K^\ast\bI)$ under uniform pricing $p_K^\ast\bI$. 
	That is, 
	\EQ{
		R(K, p^\ast_K\bI) 
		\le R(K,\bp^*).
	} 
	From Eq.~\eqref{eqn:MeanRevenue} we obtain that the optimal mean revenue is a convex combination of $(\lambda p_k\oG(p_k): k \in \cX)$, 
	where the optimal price vector is $\bp^\ast=(p_0,..,p_{K-1})$. 
	From the definition of $p^\ast_\infty$ in Eq.~\eqref{eqn:RevenueUniformPriceInfiniteServers}, 
	we get  
	\begin{equation*}
	R(K,\bp^*) 
	\le \lambda\max_{k \in \cX} p_k\oG(p_k)
	\le \lambda p_{\infty}^\ast\oG(p_{\infty}^\ast).
	\end{equation*}
	From Lemma~\ref{lem:RevenueMonotoneUniformPricing}, 
	the optimal revenue rate $R(K,p_K^*)$ is monotonically increasing in the number of servers $K$. 
	The result follows from taking the limit $K \to \infty$ in the above equation. 
\end{proof}
Thus, for a system with a large number of servers, choosing the optimal uniform price, is close to optimal. 
We note that system state for a finite server system can equivalently be represented by the number of idle servers. 
In an infinite server system, the number of idle servers is always infinite, 
and hence state dependent pricing reduces to state independent pricing. 
With this view, it is expected that optimal pricing for an infinite server system will be uniform.  
We next bound the optimal revenue rate in terms of the maximum revenue rate under uniform pricing.
\begin{lem} 
	\label{lem:UnifBnd}
	Let $p^*_\infty$ and $p^{\ast}_K$ be optimal uniform prices of infinite and finite $K$-server systems, 
	and let $\bp^{\ast}$ be the optimal state dependent price vector for the $K$ server system. 
	If the blocking probability of the $K$-server system under uniform price $p^{\ast}_\infty$ is denoted by $\pi_K(p^{\ast}_\infty)$, then 
	\EQ{
		R(K,p^{\ast}_K\bI) 
		\le R(K,\bp^{\ast}) 
		\le \dfrac{R(K,p^{\ast}_K\bI)}{1-\pi_K(p^{\ast}_\infty)}.
	}
\end{lem}
\begin{proof}
	The first inequality follows from the definition of the optimal revenue rate. 
	To prove the second inequality, recall that optimal revenue rate under uniform pricing is increasing in the number of servers, i.e. $R(K,\bp^\ast)\le\lambda p_{\infty}^*\oG(p_{\infty}^*)$. 
	Multiplying both sides by $1-\pi_K(p_{\infty}^\ast)$, we see that,
	\begin{equation*}
	(1-\pi_K(p^\ast_\infty))R(K,\bp^\ast) 
	\le 
	R(K, p^\ast_\infty\bI).
	\end{equation*}
	Since the right hand side term is the revenue rate of a $K$ server system with uniform price $p_{\infty}^*$, 
	it can be upper bounded by maximum revenue rate $R(K,p^\ast_K\bI)$ under optimal uniform price $p_K^\ast$.
\end{proof}	
The above lemma implies that the optimal revenue rate converges to maximum revenue rate under uniform pricing as the number of servers $K$ grows large. 
Further, the difference between two revenue rates decreases at least as fast as $\frac{1}{K}$.  
\begin{cor} 
	In terms of $\beta_1=\frac{1-\phi(\mu)}{\phi(\mu)}$ defined in Eq.~\eqref{eqn:beta}, 
	we can upper bound the difference between the optimal revenue rate and the maximum revenue rate under uniform pricing as 
	\begin{equation*}
	R(K,\bp^{\ast}) - R(K,p^\ast_K\bI)
	\le \frac{1}{\beta_1 K}R(K,p^{\ast}_K\bI).  
	\end{equation*}
\end{cor}
\begin{proof} 
	The blocking probability of $K$ server system given in Proposition~\ref{prop:Blocking} under uniform price $p^\ast_\infty$, can be upper bounded as 
	\begin{equation*}
	\pi_K(p_{\infty}^*)=\dfrac{1}{\sum_{j=0}^K \binom{K}{j}\oG(p_{\infty}^*)^{-j}\beta_j}\le \frac{1}{1+K\oG(p_{\infty}^*)^{-1}\beta_1}.
	\end{equation*}
	The upper bound follows by taking only two positive terms corresponding to $j \in \set{0,1}$ in the summation for $j \in \cX$. 
	Therefore, using the fact that $\oG(p_\infty) \le 1$, we get
	\begin{equation*}
	\frac{1}{1-\pi_K(p_{\infty}^\ast)} 
	\le 1+\frac{1}{K\oG(p_{\infty}^\ast)^{-1}\beta_1} 
	\le 1+\frac{1}{\beta_1K}.
	\end{equation*}
	We obtain the result by substituting this expression in the upper bound for optimal revenue rate $R(K,\bp^\ast)$ in Lemma~\ref{lem:UnifBnd}.
\end{proof}
\begin{rem}\label{rem:UnifBoundPoiss}
	For Poisson arrivals, $\beta_1=\frac{1-\phi(\mu)}{\phi(\mu)}=\frac{\mu}{\lambda}=\frac{1}{\rho}$, 
	and hence $R(K,\bp^*)\le (1+\frac{\rho}{K})R(K,p_K^\ast\bI)$. It is clear that for a large enough $K$, the optimal uniform price is a reasonable substitute for the optimal price. However, the bound is loose for smaller values of $K$ and higher values of $\rho$, corresponding to a high arrival rate.
\end{rem}

\subsection{Asymptotic Behavior of Revenue Rate}
\label{subsec:RevRateAsymptote}

We next address the question of maximum revenue rate scaling under uniform pricing as the arrival rate increases to infinity. 
For a $K$ \iid server system each serving at an exponential rate $\mu$, the maximum system service rate is $K\mu$. 
Therefore, for a uniform price system with $\bp=p\bI$, the maximum revenue cannot exceed $pK\mu$. 
We investigate whether we can meet this upper bound by driving the arrival rate to infinity. 
We observe that this is not true for all arrival distributions. 
In fact, for certain interarrival distributions, the revenue rate goes to zero as arrival rate increases. 

Recall that $\phi: \R_+ \to \R_+$ denotes the Laplace Stieltjes transform of the \iid job interarrival times. 
To begin with, we prove the following technical lemma. 
\begin{lem}
	For the $K$-server system under uniform pricing $\lim_{\lambda\to\infty}\phi(\theta) = 1$.    
\end{lem}
\begin{proof}
	For any $\theta\in \R_+$, we have $e^{-\theta U_1} \le 1$. 
	Further, we observe that $f: \R_+ \to \R_+$ defined by $f(y) \triangleq e^{-\theta y}$ is a convex function. 
	From Jensen's inequality, we have $\E f(U_1) \ge f(\E U_1)$. 
	Combining both these results, we get 
	\EQN{
		\label{eqn:LSTIneq}
		e^{-\frac{\theta}{\lambda}}= e^{-\theta\E U_1} \le \phi(\theta) \le 1,\quad \theta \in \R_+. 
	}
	Taking the limit as arrival rate $\lambda\to\infty$, we get the result.  
\end{proof}
\begin{rem}
	\label{rem:beta}
	Consider the case when $\lim_{\lambda\to\infty}\lambda(1-\phi(\mu)) = \tilde{\mu}$ exists. 
	Then, from the definition of sequence $\beta = (\beta_j = \prod_{m=1}^{j}\frac{1-\phi(m\mu)}{\phi(\mu)}: j \in \N)$ in Eq.~\eqref{eqn:beta}, 
	we get that 
	\begin{xalignat}{2}
		\label{eqn:LimitBeta}
		&\lim_{\lambda\to\infty}\beta_j = 0, &
		&\lim_{\lambda\to\infty}\lambda\beta_j = \tilde{\mu}\SetIn{j = 1}.
	\end{xalignat} 
\end{rem}
\begin{thm}
	\label{thm:AsympRevLST}
	Consider a $K$ server pricing system with job interarrival times being \iid and having a Laplace Stieltjes transform $\phi$ 
	that satisfies
	$\lim_{\lambda\to\infty} \lambda(1-\phi(\mu))=\tilde\mu$. 
	The mean revenue rate for this system under a uniform price vector $\bp=p\bI$ such that $\oG(p) >0$, 
	is bounded as the arrival rate grows. 
	In particular,  
	$\lim_{\lambda \to \infty}R(K,p\bI)=\tilde\mu p K$.
\end{thm}
\begin{proof}
	Recall that the mean revenue rate for uniform pricing $p\bI$ of $K$ \iid exponential servers is given by $R(K, p\bI) = \lambda p\oG(p)(1-\pi_K(p))$ from Eq.~\eqref{eqn:RevenueUniformPriceFiniteServers}. 
	From Proposition~\ref{prop:Blocking}, 
	we have the blocking probability $\pi_K(p)$ in Eq.~\eqref{eqn:BlockProbUnifPrice} defined in terms of variables 
	$\beta_j=\prod_{m=1}^{j}\frac{1-\phi(m\mu)}{\phi(m\mu)}$ given in Eq.~\eqref{eqn:beta} for $j \in [K]$. 
	Therefore, we can write the mean revenue rate as 
	\begin{equation}
	\label{eqn:AsympRev}
	R(K,\bp) 
	= p\oG(p)\left(\dfrac{\frac{K\lambda \beta_1}{\oG(p)}+\lambda\sum_{j=2}^K \binom{K}{j}\oG(p)^{-j}\beta_j}{\sum_{j=0}^K \binom{K}{j}\oG(p)^{-j}\beta_j}\right).
	\end{equation}
	Taking the limit as arrival rate $\lambda\to\infty$, 
	substituting the limiting results from Eq.~\eqref{eqn:LimitBeta} in Eq.~\eqref{eqn:AsympRev}, we obtain the result.
\end{proof}
\begin{rem}
	From the inequality on Laplace Stieltjes transform $\phi$ in Eq.~\eqref{eqn:LSTIneq} and the fact that $1- y \le e^{-y}$, we get
	\begin{equation*}
	0 \le \lambda(1-\phi(x))\le \lambda(1-e^{-\frac{x}{\lambda}})\le x.
	\end{equation*}
	From the definition of $\tilde{\mu} = \lim_{\lambda\to\infty}\lambda(1-\phi(\mu))$, 
	we obtain that $0 \le \tilde{\mu} \le \mu$. 
	Thus, depending on the inter arrival time distribution, 
	the limiting revenue can lie between $0$ and $\mu pK$. 
	The quantity $\tilde\mu$ can be considered an asymptotic service rate per server.
\end{rem}
We present examples of limiting revenue rate being $\mu pK$, zero, and between $(0, \mu pK)$ in Examples~\ref{exmp:Max},~\ref{exmp:Zero}, and~\ref{exmp:Inter} respectively, in Appendix~\ref{subsec:FixUnifPrice}.   
We see that with the fixed uniform pricing, the limiting mean revenue rate remains bounded, 
even when the arrival rate $\lambda$ increases unboundedly large. 
We next show that it is indeed possible to scale the mean revenue rate with the arrival rate, 
at least in the limiting regime, 
if the uniform pricing scales with the job arrival rate $\lambda$. 
\begin{lem}
	Consider a $K$ server uniform pricing system with \iid job interarrival times having Laplace Stieltjes transform $\phi: \R_+\to \R_+$. 
	If the limit $\tilde\mu\triangleq \lim_{\lambda\to\infty}\lambda(1-\phi(\mu)) > 0$, 
	the value distribution $G$ 
	has the support $\R_+$, 
	and the uniform price $p\in\oG^{-1}(\frac{1}{\lambda})$, 
	then the limiting revenue rate  $\lim_{\lambda \to \infty} R(K, p\bI) = \infty$.  
\end{lem}
\begin{proof}
	Let $p \in \oG^{-1}(1/\lambda)$. 
	Substituting $\oG(p) = \frac{1}{\lambda}$ in the mean revenue rate in Eq.~\eqref{eqn:RevenueUniformPriceFiniteServers} for $K$ server system under uniform pricing,   
	we obtain $R(K,p\bI) = p(1-\pi_K(p))$. 
	From the blocking probability $\pi_K(p)$ expression in Eq.~\eqref{eqn:BlockProbUnifPrice} in terms of the positive sequence $\beta = (\beta_j: j \in [K])$ in Eq.~\eqref{eqn:beta}, 
	we get 
	\begin{equation*}
	\pi_K(p)=\dfrac{1}{\sum_{j=0}^K \binom{K}{j}\oG(p)^{-j}\beta_j}\le \frac{1}{1+K\oG(p)^{-1}\beta_1},
	\end{equation*}
	Recall that $\phi(\mu) \le 1$, and hence $\beta_1 = \frac{(1-\phi(\mu))}{\phi(\mu)} \ge 1-\phi(\mu)$. 
	Using this fact and substituting $\oG(p) = 1/\lambda$ in the above equation, 
	we get 
	$\pi_K(p) \le (1+ K\lambda(1-\phi(\mu))^{-1}$.
	From the hypothesis $\lim_{\lambda\to\infty}\lambda(1-\phi(\mu)) = \tilde\mu> 0$, 
	and the fact that $\lim_{x\to 0}\oG^{-1}(x) = \infty$\footnote{From the definition of distribution functions, the complimentary distribution $\oG$ is non-increasing and $\lim_{x\to\infty}\oG(x) = 0$.  
		Further, 	since the support of $G$ is $\R_+$, it follows that $\lim_{x\to 0}\oG^{-1}(x) = \infty$.    
	}, 
	we obtain 
	\EQ{
		\lim_{\lambda\to\infty}R(K,p\bI) 
		= \lim_{\lambda\to 0}\oG^{-1}\left(\frac{1}{\lambda}\right)\frac{\tilde\mu}{1+\tilde\mu} = \infty.
	}
\end{proof}
Thus, an arrival rate dependent uniform pricing can scale the revenue rate to infinity, in the asymptotic regime as the arrival rate $\lambda$ grows arbitrarily large. 
We show an example of linear increase of mean revenue rate with arrival rate $\lambda$ in Example~\ref{exmp:Linear}.   
Since $\oG^{-1}(\frac{1}{\lambda}) \to \infty$ as $\lambda$ increases, 
we see that to extract maximum revenue, 
the price should be made as high as possible in the heavy traffic limit. 
However, letting the price grow too fast can cause the revenue rate to go to zero instead of infinity, 
as shown in Example~\ref{exmp:Log}. 

\section{Optimal Pricing for Finite Servers with Poisson Arrivals}
\label{MDP-finite-servers}
In the previous section, we found the optimal uniform pricing for a finite server system. 
Uniform pricing is optimal when the number of servers is very large. 
However, this yields a sub optimal revenue rate when the number of servers is finite. From Remark~\ref{rem:UnifBoundPoiss}, it seems that uniform pricing would be sub optimal in a system with few servers or with high load, i.e, arrival rate much higher than service rate.
In order to compute the revenue maximizing price, we frame the optimal state dependent pricing problem as a continuous time Markov decision problem~\cite[Chapter~5]{bertsekas2007}. 
We derive optimal prices and also analyze their dependence on various system parameters, 
e.g., the number of servers, job arrival rate, and service rate. We first formulate the MDP for the case of Poisson arrivals, and solve it. In the subsequent section, we solve the MDP for a system with general arrivals. These are dealt with separately because the formulation changes when we move from Poisson to general arrivals. Furthermore, under the Poisson assumption e are able to obtain more insights into the system behaviour.

\subsection{The MDP formulation}
\label{sec:mdp-formulation}
As in Section~\ref{sec:sys-model}, 
we consider the number of busy servers to be the state of the system and the quoted price in any state to be the control. 
Correspondingly, the state space is $\cX^\prime$ and the control space for price $u \in \R_+^{\cX^\prime}$. 
The mean revenue rate given a stationary state dependent policy $u$ is,
\begin{equation}\label{eqn:MdpAvgCost}
R(K,u)= \lim_{t\to\infty}\frac{\E R(t)}{t}= \lim_{t\to\infty}\frac{1}{t}\E[\int_{0}^{t} g(X_s,u(X_s))ds], 
\end{equation} 
where $g$ is the instantaneous reward. In our setup, the rewards are obtained only at the arrival instants, and equals the price $u$ if accepted by the incoming arrival. 
However, the price $u$ is changed at every transition instant. 
Denoting $X_n$ as the state of the system after $n$ transitions, we can rewrite the reward rate as
\begin{equation*}
R(K,u)= \lim_{t\to\infty}\frac{1}{t}\E \sum_{n=0}^{N_t}g(X_n,u(X_n)).
\end{equation*}

Following the discussion in~\cite{bertsekas2007}, this is equivalent to
\begin{equation*}
R(K,u)=\lim_{N \to \infty}\frac{1}{\E t_N} \E\sum_{n = 1}^{N} g(X_n,u(X_n)),
\end{equation*}
where $t_N$ is the $N$-th transition epoch. 
We wish to find the control $u^*\in \R_+^{\cX^\prime}$ that yields the optimal reward rate $R_K^*=R(K,u^*)=\max_{u}R(K,u)$.

The sojourn times in various states are independent exponentially distributed random variables  depending on the controls applied on transitions to those states. As soon as the state changes to state $i$, a price $u$ is set. 
This price is accepted with probability $\oG(u)$ by an incoming arrival. 
Therefore, the sojourn times in a state $i$, for price $u$, are exponentially distributed with parameters  $\nu_i(u) = i\mu + \lambda \oG(u)\SetIn{i \in \cX'}$. 
The state transition probabilities are independent of the sojourn times and dependent on the price $u \in \R_+$, and are given by: 
$p_{0,1}(u) = 1$ and $p_{K,K-1}(u) = 1$, 
and for $i \in [K-1]$ 
\EQN{
	\label{eqn:TransProb}
	p_{ij}(u) 
	= \frac{\lambda \oG(u)}{\nu_i(u)}\SetIn{j = i+1} 
	+ \frac{i\mu}{\nu_i(u)}\SetIn{j = i-1}. 
}
In addition, the rewards are accrued at the state transition instants and hence we focus only on the embedded discrete time Markov chain. 
The duration between two transitions is referred to as a stage of the MDP.  
When in a state $i$ and using control $u$, 
a single stage reward $u$ is obtained if a job arrives and joins service leading to the state~$i+1$.  
The mean single stage reward is  
\EQN{
	\label{eqn:g-i-u}
	g(i,u)  = u p_{i,i+1}(u)=  u\SetIn{i = 0} + \frac{\lambda u\oG(u)}{\nu_i(u)}\SetIn{i \in [K-1]}.
}
\subsection{Uniformization of continuous time Markov chain}
\label{sec:uniformization}
Using~\cite[Proposition  5.3.1]{bertsekas2007} to solve the average reward MDP in Eq.~\eqref{eqn:MdpAvgCost} we can write  the Bellman's equation 
for all states $i$
\EQN{
	\label{eqn:bellman-continuous}
	h(i) = \max_u\Big[g(i,u) - \frac{\theta}{\nu_i(u)} + \sum_{j=0}^K p_{ij}(u)h(j) \Big]. 
}
Here  $\theta$ is the optimal average reward per stage  independent of the initial state~(see~\cite[Section~4.1]{bertsekas2007}) and $h(i)$, has interpretation of a relative or differential reward for each state $i$. 
Defining the uniformizing transition rate $\Lambda \triangleq K\mu + \lambda$, we observe that $\nu_i < \Lambda$ for all states $i$ and control $u \in \R_+$. 
Hence we can convert the above Markov controlled process to the one with uniform transition rate $\Lambda$ by allowing fictitious self transitions such that the resulting dynamics remains unchanged. Specifically, we redefine state transition probabilities for the uniformized Markov process as follows.   
For all states $i \in \cX$ and control $u \in R_+$, 
\EQN{
	\label{eqn:TransProbUniform}
	\tp_{ij}(u) = 
	p_{ij}(u)\frac{\nu_i(u)}{\Lambda}\SetIn{j \neq i}
	+ \Big(1- \frac{\nu_i(u)}{\Lambda}\Big)\SetIn{j = i}. 
}
We can now view the above problem as a discrete-time average reward problem with same state and control spaces, transition probabilities $\tp_{ij}(u)$ and expected single stage rewards $g(i,u)$. The Bellman's equation for this discrete-time problem has the following form for all $i$
\EQN{
	\label{eqn:bellman-discrete}
	\th(i) = \max_u\Big[g(i,u)\nu_i(u) - \theta + \sum_{j=0}^K \tp_{ij}(u)\th(j)\Big].
}
\begin{rem}
	The Bellman's equations~\eqref{eqn:bellman-continuous} and~\eqref{eqn:bellman-discrete} are equivalent. 
	In particular, a pair $(\theta, h)$ satisfies~\eqref{eqn:bellman-continuous} 
	if and only if the pair $(\theta,\th)$ satisfies~\eqref{eqn:bellman-discrete}, 
	where $\th(i) = \Lambda h(i)$ for all $i$. 
	Moreover, for all the states, the optimal actions for the two problems~(control $u$ achieving maxima in the right hand sides of~\eqref{eqn:bellman-continuous} and~\eqref{eqn:bellman-discrete}) are identical.   
\end{rem}
\begin{rem}
	Defining the uniformized reward difference $\Delta(i) \triangleq \frac{(\th(i) - \th(i+1))}{\Lambda}$ for all states $i\in \cX'$,
	and substituting in Eq.~\eqref{eqn:bellman-discrete}, 
	along with expressions for per stage mean reward $g(i,u)$ from Eq.~\eqref{eqn:g-i-u}, 
	and transition probabilities $\tp_{ij}(u)$ from Eq.~\eqref{eqn:TransProbUniform}, 
	we get the following set of equations for all $i \in \cX$
	\EQN{
		\label{eqn:Delta}
		\theta = \lambda\max_u\set{\oG(u)(u - \Delta(i))}\SetIn{i \in \cX'} + i\mu\Delta(i-1). 
	}
\end{rem}

\subsection{Auxiliary maps}
\label{sec:aux}
We define the mapping $f: \R^2 \to \R$ as 
\EQN{
	\label{eqn:def-f}
	f(B,u) \triangleq (u-B)\oG(u),~B, u \in \R. 
}
\begin{rem}
	We define a set valued map $u^\ast$ that maps $B\in \R$ to $u^\ast(B)\subseteq \R_+$ 
	\begin{equation}
	\label{eqn:DefnArgMax}
	u^\ast(B) \triangleq \arg\max_uf(B,u). 
	\end{equation}
	If the maximizer is unique, then $u^\ast: \R \to \R_+$ is a real valued map. 
	The maximum value of $f(B,u^\ast)$ is a real valued map $m: \R \to \R$ such that 
	\begin{equation}
	\label{eqn:def-m}
	m(B) \triangleq f(B, u^\ast)=\max_u f(B,u).
	\end{equation}
\end{rem}
\begin{lem} 
	\label{lem-m-monotonicity} 
	Following statements are true for $m$ and $u^\ast$. 
	\begin{enumerate}[(a)]
		\item $m$ is non-negative and decreasing in $B$.
		\item $m$ is Lipschitz-1 continuous and convex function of $B$. 
		\item For $B_1 < B_2$, we have $\sup u^\ast(B_1) \le \inf u^\ast(B_2)$. 
		When $f(B,u)$ has a unique maximizer in $u$, 
		then this maximizer $u^\ast$ is non-decreasing in $B$. 
	\end{enumerate}
\end{lem}
\begin{proof}
	Proof is in Appendix~\ref{sec:ProofMMonotone}.
\end{proof}
\remove{
	We assume that in the above problems maxima are achieved and $m(B)$ are finite for all $B$. In particular, we assume that $\lim_{u \to \infty} u \bg(u) < \infty$.\footnote{Notice that, for any $B$, $m(B) \geq \max_u u\bg(u) _ B \geq \lim_{u \to \infty} u \bg(u) - B$. Hence $\lim_{u \to \infty} u \bg(u) = \infty$ implies $m(B) = \infty$. Moreover, $\lim_{u \to \infty} u \bg(u) = \infty$ also implies that the expected valuation is infinite.}
}	
\subsection{The optimal pricing}
\label{sec:opt-pricing}

In terms of the map $m$, we can re-write the Eq.~\eqref{eqn:Delta} as 
\EQN{
	\label{eqn:iterative}
	m(\Delta(i))\SetIn{i \in \cX'}+ \frac{i\mu}{\lambda}\Delta(i-1)= \frac{\theta}{\lambda},~ i \in \cX.
}
Observe that if $(\Delta^\ast(i):i\in\cX')$ solves Eq.~\eqref{eqn:iterative} 
then the control $u^\ast_i \triangleq u^\ast(\Delta^\ast(i))$ achieving $m(\Delta^\ast(i))$ in Eq. ~\eqref{eqn:def-m} 
is the optimal control in each state $i \in \cX'$.
\begin{rem}
	Consider the limiting case of  infinitely many servers, i.e., $K = \infty$. 
	We can easily see that $\theta = \lambda m(0)$ along with $\Delta(i) = 0$ for all $i \in \Z_+$ is a solution to  Eq.~\eqref{eqn:iterative}. 
	In particular, uniform~(state independent) pricing, $u^\ast = \arg\max_{u \geq 0} u\bg(u)$, achieves the optimal revenue rate as readily seen in Eq.~\eqref{eqn:RevenueUniformPriceInfiniteServers}.
\end{rem}

\begin{lem}
	\label{lem-delta-monotonicity}
	Let $(\theta, \Delta(i), i \in \cX')$ be a solution to Eq.~\eqref{eqn:iterative} 
	and $\bp^\ast_K = (u_0^\ast, \dots, u^\ast_{K-1}) \in \R_+^{\cX'}$ be the optimal price vector. 
	Then 
	\begin{enumerate}[(a)]
		\item $\theta \ge 0$,
		\item $\Delta(i)$ are positive and increasing in $i \in \cX'$.
		\item $u^\ast_i$ are also increasing in $i \in \cX'$.
	\end{enumerate}
\end{lem}
\begin{proof} 
	Proof is in Appendix~\ref{sec:ProofDeltaMonotone}.
\end{proof}

Next, we will focus on solving Eq.~\eqref{eqn:iterative}. 
We propose an iterative algorithm to obtain $\theta$, 
which can then be used to obtain $\Delta(i)$ and also the optimal price $u^\ast_i$ for all the states.  
Realizing that $\Delta(i)$ is a function of optimal revenue $\theta$ and state $i$, 
we denote it as $g_i(\theta) \triangleq \Delta(i)$, to rewrite Eq.~\eqref{eqn:iterative} as 
\begin{subequations}
	\begin{align}
	\label{eqn:iterative2a}
	\theta &= \lambda m(g_0(\theta)),\\
	\label{eqn:iterative2b}
	g_{i-1}(\theta) &= \frac{\theta - \lambda m(g_i(\theta))}{i\mu},~ i \in [K-1], \\
	\label{eqn:iterative2c}
	g_{K-1}(\theta) &=\frac{\theta}{K\mu}. 
	\end{align}
\end{subequations}
We will show that there exists a unique $\theta$ which solves  Eq.~\eqref{eqn:iterative2a}.
We then propose   Algorithm~\ref{algo:iterative-algo} that finds this unique $\theta$ in terms of which the optimal prices can be found. In particular, this algorithm iteratively generates two sequences $(\ubar{\theta}_k: k \in \cX)$ and $(\bar{\theta}_k: k \in \cX)$ which converge to the unique $\theta$.       
\begin{algorithm}[hhhh]
	\caption{}\label{algo:iterative-algo}
	\begin{algorithmic}[0]
		\State {\bf initialize} $k = 0,\ubar{\theta}_0 = 0, \bar{\theta}_0 = \lambda m(g_0(0))$,
		\While{$\bar{\theta}_k - \ubar{\theta}_k > \delta$}\ \ \  \Comment{$\delta$ is the desired precision.}
		\State $\tilde{\theta}_k = \frac{\ubar{\theta}_k + \bar{\theta}_k}{2}$,
		\State $\ubar{\theta}_{k+1} = \max\left\{\ubar{\theta}_k,\min\{\tilde{\theta}_k,\lambda m(g_0(\tilde{\theta}_k))\}\right\}$,
		\State $\bar{\theta}_{k+1} = \min\left\{\bar{\theta}_k,\max\{\tilde{\theta}_k,\lambda m(g_0(\tilde{\theta}_k))\}\right\}$,
		\State $k = k+1$
		\EndWhile
	\end{algorithmic}
\end{algorithm}

\begin{thm}
	\label{thm:fixed-point-eqn}
	\begin{enumerate}[(a)]
		\item The fixed point equation $\theta = \lambda m(g_0(\theta))$ has unique solution.
		\item In Algorithm~\ref{algo:iterative-algo}, $\ubar{\theta}_k \uparrow \theta^\ast$ and $\ubar{\theta}_k \downarrow \theta^\ast$, where $\theta^\ast$ is the unique fixed point.
	\end{enumerate}
\end{thm}
\begin{proof} 
	We consider the Eqs.~\eqref{eqn:iterative2a},~\eqref{eqn:iterative2b},~\eqref{eqn:iterative2c}.
	\begin{enumerate}[(a)]
		\item Observe that $\lambda m(g_0(0)) > 0$. 
		We now argue that 
		$\lambda m(g_0(\theta))$ is decreasing in $\theta$. These two facts together yield both existence and uniqueness. 
		From the monotonicity of function $m$ in Lemma~\ref{lem-m-monotonicity}(a) and definition of $g_{i-1}$ from Eq.~\eqref{eqn:iterative2b}, 
		it follows that $g_{i-1}$ is increasing in $\theta$ if $g_i$ is increasing in $\theta$. 
		Since $g_{K-1}(\theta) = \theta/K\mu$ is increasing in $\theta$, 
		it follows that $g_0(\theta)$ is increasing in $\theta$, 
		and hence $\lambda m(g_0(\theta))$ is decreasing in $\theta$. 
		
		\item See~\cite[Theorem~2.1]{divya-singh17offloading}.
	\end{enumerate}
\end{proof} 
\begin{rem}
	If we assume the valuation is exponentially distributed, i.e., $\oG(x)=e^{-\beta x}$, we see that the mapping $u^*(B)=B+\frac{1}{\beta}$ and $m(B)=\frac{1}{\beta}e^{-(\beta B+1)}$. The optimal prices will be $u_i^*=u^*(\Delta^*(i))=\Delta^*(i)+\frac{1}{\beta}.$
\end{rem}

\subsection{Properties of the Optimal Solution}
\label{sec:prop-opt-solution}

We now analyze how the optimal prices and the optimal  revenue rate  vary with arrival rate $\lambda$, service rate $\mu$, and number of servers $K$. 
We use the fact that the optimal revenue rate $\theta^\ast$ is solution to the Eq.~\eqref{eqn:iterative2a}, 
from which we inductively derive properties of the uniformized reward difference $g_i$ using the monotonic decrease of $m$ from Lemma~\ref{lem-m-monotonicity}. First, we look at the variation of the optimal revenue rate with arrival rate $\lambda$.
\subsubsection{Varying Arrival Rate}
We assume that we vary the arrival rate $\lambda$ while keeping the service rate $\mu$ and number of servers $K$ fixed.
\begin{prop}
	\label{prop:VarRevenueArrival}
	For a $K$ server pricing system with a fixed service rate $\mu$, the following statements are true.
	\begin{enumerate}[(a)]
		\item The optimal revenue rate $\theta^\ast(\lambda)$ increases with $\lambda$.
		\item The ratio $\theta^\ast(\lambda)/\lambda$ decreases with $\lambda$. 
	\end{enumerate}
\end{prop}	
\begin{proof} 
	Proof is in Appendix~\ref{sec:ProofRevenueArr}.
\end{proof}
\remove{	
	Since $g_{K-1}(\theta^\ast(\lambda))$ and $g_0(\theta^\ast(\lambda)) = m^{-1}(\theta^{\ast}(\lambda)/\lambda)$ increase with $\lambda$, so do $u^\ast_{K-1}(\lambda)$ and $u^\ast_0(\lambda)$.
	\textcolor{red}{Do the optimal prices for all the states increase with $\lambda$?}
}
\remove{   
	Since $\theta^\ast(\lambda)/\lambda$ decreases with $\lambda$, $g_{0}(\theta^\ast(\lambda))$ increases. Once more iteratively 
	arguing,
	\[
	m(g_i(\theta^\ast(\lambda)) = \frac{\theta^\ast(\lambda)}{\lambda} - i\mu
	\]
}
\begin{rem}\label{rem:LamVarRev}
	The uniformized reward differences $\Delta(0)$ and $\Delta(K-1)$ are increasing in the arrival rate  $\lambda$ from Eqns.~\eqref{eqn:iterative2a} and~\eqref{eqn:iterative2c}, respectively. Consequently, the optimal prices $u_0^{\ast}$ and $u_{K-1}^{\ast}$ are also increasing in $\lambda$. We believe that all the optimal prices $(u_i^{\ast},i\in\cX')$ are increasing in $\lambda$. While we have not been able to show this, we demonstrate it via numerical results in Section~\ref{subsec:NumRes}.
\end{rem}
Next, we study how the optimal revenue rate varies as the service rate $\mu$ changes.
\subsubsection{Varying Service Rate}
Here we assume that we vary the service rate $\mu$ while keeping the arrival rate $\lambda$ and number of servers $K$ fixed. Now we express the revenue rate as $\theta^\ast(\mu)$ to emphasize its dependence on $\mu$.
\begin{prop}
	\label{prop:VarRevenueService}
	For a $K$-server pricing system with a fixed arrival rate $\lambda$, the following statements are true.
	\begin{enumerate}[(a)]
		\item The revenue rate $\theta^\ast(\mu)$ increases with $\mu$. 
		\item The ratio $\theta^\ast(\mu)/\mu$ decreases with $\mu$.
	\end{enumerate}
\end{prop}
\begin{proof}
	Proof is in Appendix~\ref{sec:ProofRevenueService}.
\end{proof}
\remove{	
	Since $g_0(\theta^\ast(\mu))$ and $g_{K-1}(\theta^\ast(\mu)) = \theta^{\ast}(\mu)/K\mu$ decrease with $\mu$, so do $u^\ast_0(\mu)$ and $u^\ast_{K-1}(\mu)$. \textcolor{red}{Do the optimal prices for all the states decrease with $\mu$?}
}
\begin{rem}\label{rem:MuVarRev}
	Contrary to the observation in Remark~\ref{rem:LamVarRev}, the uniformized reward differences $\Delta(0)$ and $\Delta(K-1)$ are decreasing in the service rate  $\mu$ from Eqns.~\eqref{eqn:iterative2a} and~\eqref{eqn:iterative2c}, respectively. Hence, the optimal prices $u_0^{\ast}$ and $u_{K-1}^{\ast}$ are decreasing in $\mu$. We believe that all the optimal prices $(u_i^{\ast},i\in\cX')$ are decreasing in $\mu$. While we have not been able to show this, we demonstrate it via numerical results in Section~\ref{subsec:NumRes}.
\end{rem}	
Finally, we study the variation of optimal revenue rate with number of servers.
\subsubsection{Increasing number of servers}    
We assume that we vary number of servers $K$ while keeping arrival rate $\lambda$ and service rate $\mu$ fixed. Now We express the revenue rate as $\theta^\ast(K)$.
\begin{prop}
	\label{prop:VarRevenueServers}
	For a  pricing system with a fixed arrival rate $\lambda$ and a fixed service rate $\mu$, the following statements are true.
	\begin{enumerate}[(a)]
		\item The revenue rate $\theta^\ast(K)$ increases with $K$.
		\item The ratio $\theta^\ast(K)/K$ decreases with $K$. 
		\item For any $i< K$, the optimal price $u^\ast_i(K)$ is non-increasing with $K$. 
	\end{enumerate}
\end{prop}
\begin{proof} 
	Proof is in Appendix~\ref{sec:ProofRevenueServers}.
\end{proof}	
\subsubsection*{General Service Times}
We make an interesting observation on multiple server systems with Poisson job arrivals with rate $\lambda$ and general \iid job service times with distribution $F: \R_+  \to [0,1]$ and mean $\frac{1}{\mu}$. 
Suppose we continue to use the optimal state dependent prices $u^\ast_i$ for $i \in \cX^\prime$ busy servers seen by an incoming arrival, 
that was derived for the exponential service rate system in Section~\ref{sec:opt-pricing}. 
This results in an $M/G/K/K$ system with state dependent arrivals rates $(\lambda_i \triangleq u^\ast_i \lambda: i \in \cX^\prime)$.  
Following the insensitivity property~\cite[Section 8.10]{zukerman2013introduction} of $M/G/K/K$ systems, 
the steady state distribution of the number of busy servers remain identical to 
the steady state distribution in the corresponding $M/M/K/K$ system. 
Moreover, the average reward rate in the $M/G/K/K$ system with state dependent prices $(u^\ast_i: i \in \cX^\prime)$  will be same as the optimal average reward rate in the $M/M/K/K$ system. 
However, the optimal prices in the $M/G/K/K$ systems will in general be different from $(u^\ast_i:  i \in \cX^\prime)$. The optimal prices will depend on elapsed services times of busy servers on job arrival epochs. 
These optimal prices are not easy to determine following the techniques as used in this work. 
However, we make a non-trivial inference that the optimal average reward rate in an $M/G/K/K$ system always exceeds the optimal average reward rate in the corresponding $M/M/K/K$ system.

\section{General Arrival Processes} 
\label{sec:GenArrival}
In the previous section, we found the optimal pricing for a $K$ server system with Poisson arrivals and exponential service rates.  
In this section, we extend the setting to $K$ server systems with general interarrival time distribution. 
In particular, we assume the interarrival times $(U_n: n \in \N)$ are \iid with density $f: \R_+ \to \R_+$ and finite mean $1/\lambda$.  
We will continue to assume that the admission price is updated only at arrival instants, 
and hence this price depends only on the number of busy servers in the system. 
We assume that the price is infinite when all $K$ servers are empty. 
As discussed in Section~\ref{section:ComputeRevRate}, 
the system state is modeled by the the number of occupied servers seen by the arriving jobs, 
and the state space remains $\cX = \set{0,\dots,K}$, 
and the modified state space $\cX^\prime = \set{0,\dots, K_1}$. 
Similarly, the control space for price remains $\R_+^{\cX^\prime}$, 
and we write the problem of finding optimal revenue rate as an MDP. 
In the Poisson arrival setting, 
the process $X = (X(t): t \ge 0)$ sampled at all transition instants, remained Markov. 
In contrast, in the general arrival setting, the process $X$ sampled only at the arrival instants, is Markov. 
Thus the sampled process $Z = (Z_n = X(A_n^-): n \in \N)$ is a controlled Markov chain. We modify the MDP in Section~\ref{sec:mdp-formulation}, 
to write the optimal revenue rate in the terms of sampled process $Z$, 
and the instantaneous reward at arrival instants $g(Z_n, u(Z_n)) = \E[u(Z_n)\SetIn{V > u(Z_n)}|Z_n] = u(Z_n)\oG(Z_n)$, as  
\begin{equation*}
R(K,u) = \lim_{N\to\infty}\frac{1}{t_N}\E\sum_{n=1}^{N}g(Z_n, u(Z_n)).     
\end{equation*}
The probability of $k-j$ departures from state $k$ is given by $\alpha_{k,k-j}$ defined in Eq.~\eqref{eqn:NumDeparturesSampledProcess}. 
We recall the transition probability from state $k \in \cX$ to state $j \in \set{0, \dots, \min\set{k+1,K}}$ for the controlled Markov chain $Z$ given in Eq.~\eqref{eqn:TransitionSampledProcess}, 
with price $p_k$ replaced by control map $u$ is
\begin{equation*}
p_{kj}(u) = \oG(u)\alpha_{k+1,k+1-j} + G(u)\alpha_{k,k-j}. 
\end{equation*}
Following similar steps as in Section~\ref{sec:uniformization}, 
we use~\cite[Proposition  5.3.1]{bertsekas2007} to solve the average reward MDP in the above equation. 
We can write the Bellman's equations for all states $i$
\EQ{
	h(i) = \max_u\Big[g(i,u) - \frac{\theta}{\lambda} + \sum_{j=0}^K p_{ij}(u)h(j)\Big],\quad i \in \cX^\prime.
}
Note that the mean sojourn time $\frac{1}{\nu(i)} = \frac{1}{\lambda}$ for all states~$i$, 
and $\theta$ is the optimal average reward per stage, independent of the initial state. 
Substituting the instantaneous reward $g(i, u) = u\oG(u)$ at arrival instants, 
the transition probabilities for the sampled Markov chain $Z$ in Eq.~\eqref{eqn:TransitionSampledProcess}, 
and the probability distribution of number of departures between two arrival instants in Eq.~\eqref{eqn:NumDeparturesSampledProcess}, 
we get 
\begin{equation}
\label{eq:bellman-gm11}
\begin{split}
h(i) &= \max_u\Big[u \oG(u) - \frac{\theta}{\lambda} + \oG(u)\sum_{j=0}^{i+1} \alpha_{i+1,i+1-j}h(j)\\
&+ G(u)\sum_{j=0}^{i} \alpha_{i,i-j}h(j)\Big], \quad i \in \cX^\prime.
\end{split}
\end{equation}
When the number of busy servers is $K$, we get the boundary equation  
\begin{equation}
\label{eq:bellman-gm12}
h(K) =  - \frac{\theta}{\lambda} + \sum_{j=0}^{K} \alpha_{K, K-j}h(j). 
\end{equation}
We first focus on states $i \in [K-1]$. 
To this end, we define the reward difference 
\EQN{
	\Delta(i) \triangleq h(i) - h(i+1),\quad i \in [K-1],
}
and the probability of more than $i-j$ departures from state $i$ as  
\EQN{ 
	a_{i,j} \triangleq \sum_{l=0}^j\alpha_{i,i-l},\quad i \in \cX, j \le i.
}  
Rearranging the terms in Eq.~\eqref{eq:bellman-gm11}, 
using the definition of sequences $(\Delta(i): i \in [K-1])$ and $(a_{i,j}: j \le i, i \in \cX)$,  
we get
\begin{equation*}
\begin{split}
&\max_u\Big[\big(u - \sum_{j=0}^{i-1}(a_{ij} - a_{i+1,j})\Delta(j) - \alpha_{i+1,0}\Delta(i)\big)\bar{G}(u)\Big]\\
&+ \sum_{j=0}^{i-1} a_{ij} \Delta(j) = \frac{\theta}{\lambda}.
\end{split}
\end{equation*}
Following similar steps for 
for $i=K$ in Eq.~\eqref{eq:bellman-gm12}, 
we get 
\begin{equation*}
\sum_{j=0}^{K-1}a_{K,j} \Delta(j) = \frac{\theta}{\lambda}.
\end{equation*}
For notational convenience, we define the following sequence 
\begin{equation}
\label{eqn:ArgMax}
b_i \triangleq 
\sum_{j=0}^{i-1} (a_{i,j} - a_{i+1,j})\Delta(j) + \alpha_{i+1,0}\Delta(i),\quad i \in \cX^\prime. 
\end{equation}
From the definition of map $m(B) = \max_u (u-B)\oG(u)$ defined in Eq.~\eqref{eqn:def-m} 
and the definition of $(b_i: i \in \cX^\prime)$ in Eq.~\eqref{eqn:ArgMax}, 
we can write the previous set of equations for the solution of average reward MDP as 
\begin{subequations}
	\begin{align}
	\label{eq:bellman-gm1}
	m(b_0) & = \frac{\theta}{\lambda}, \\
	\label{eq:bellman-gm2}
	m(b_i) + \sum_{j=0}^{i-1} a_{i,j} \Delta(j) & = \frac{\theta}{\lambda}, i \in [K-1], \\
	\label{eq:bellman-gm3} 
	\sum_{j=0}^{K-1} a_{K,j} \Delta(j) &= \frac{\theta}{\lambda}.
	\end{align}
\end{subequations}
\begin{thm}
	\label{thm:DeltaMonotonicityG}
	Let $(\theta, (\Delta(i): i \in \cX^\prime))$ be a solution to Eqs.~\eqref{eq:bellman-gm1}-\eqref{eq:bellman-gm3}. 
	Then, the following statements hold true.  
	\begin{enumerate}[(a)]
		\item The optimal revenue  rate $\theta \ge 0$. 
		\item The reward rate difference sequence $(\Delta(i), i \in \cX^\prime)$ is positive and increasing in state $i$.  
		The sequence $(b_i: i \in \cX^\prime)$ is also positive and increasing in state $i$. 
		\item The optimal price vector $(u^\ast_i: i \in \cX^\prime)$ is increasing in state $i$. 
	\end{enumerate}
\end{thm}
\begin{proof}
	Proof is in Appendix~\ref{sec:GenArrMDP}.
\end{proof}
When the inverse map $m^{-1}$ exists, we provide an inductive procedure to get a fixed point equation to obtain the optimal state dependent mean revenue rate $\theta$. 
Given the optimal mean revenue rate $\theta$, 
the reward difference $\Delta(i)$ and hence the optimal actions $u^\ast_i$ can be obtained for all states $i \in \cX^\prime$.  
To show explicit dependence of the reward difference on the mean revenue rate $\theta$, we denote the reward difference $\Delta(i) = g_i(\theta)$ for $i \in \cX^\prime$. 
Substituting this in Eqs.~\eqref{eq:bellman-gm1}-\eqref{eq:bellman-gm2}, we can inductively obtain 
\begin{align*}
g_{0}(\theta) =& 
\frac{1}{\alpha_{1,0}}m^{-1}\left(\frac{\theta}{\lambda}\right),    \\ 
g_i(\theta) =& \frac{1}{\alpha_{i+1,0}}\Big[(m^{-1}\big(\frac{\theta}{\lambda} 
- \sum_{j=0}^{i-1}a_{i,j} g_j(\theta)\big)\\
&-  \sum_{j=0}^{i-1} (a_{i,j} - a_{i+1,j})g_j(\theta)\Big],\quad i \in \cX^\prime\setminus\set{0}.
\end{align*}
Finally, using the sequence of functions $(g_j(\theta): j \in \cX^\prime)$ to replace reward difference $(\Delta(j): j \in \cX^\prime)$ in Eq.~\eqref{eq:bellman-gm3}, we obtain the following fixed point equation  
\begin{equation}
\label{eq:theta-fpe-gm}
\theta = \lambda \sum_{j = 0}^{K-1}a_{K,j} g_j(\theta). 
\end{equation}
We can solve the fixed point equation in Eq.~\eqref{eq:theta-fpe-gm} to obtain the optimal mean revenue rate $\theta$ and the optimal prices $(u_i^\ast: i \in \cX^\prime)$.

\section{Numerical Evaluation}\label{subsec:NumRes}
We first obtain the optimal price vectors for a $5$ server system, for three different inter arrival time  distributions. The arrival rate is $\lambda=25$ and service rate $\mu=2$ for all these systems, and the valuation distribution is $\oG(p)=e^{-p}$. The optimal price as a function of the number of busy servers, for exponential, uniform and constant inter arrival time distributions, are plotted in fig.~\ref{fig:OptPriExp}. Note that the mean inter arrival time will be $\frac{1}{\lambda}$.  All the prices are increasing in the number of busy servers, as shown before in Lemma~\ref{lem-delta-monotonicity}(c) and Theorem~\ref{thm:DeltaMonotonicityG}(c). Also observe that the exponential inter arrival time attracts the highest price in any system state.
\begin{figure}
	\centering 
	\begin{tikzpicture}[scale=0.6]
		\begin{axis}[
		x tick label style={
			/pgf/number format/1000 sep=},
		ylabel=Price,ymin=0.4,ymax=1.8,
		xlabel=number of busy servers,
		enlargelimits=0.15,
		ybar,
		bar width=2pt,
		legend style={at={(0.5,1.0)},
			anchor=north,legend columns=3},
		]
				
		\addplot 
		coordinates {(0,1.1515) (1,1.1923)
			(2,1.2593) (3,1.3865) (4,1.7041)};
		\addplot 
		coordinates {(0,1.16) (1,1.2028)
			(2,1.2729) (3,1.4050) (4,1.7312)};
		\addplot 
		coordinates {(0,1.17) (1,1.22)
			(2,1.29) (3,1.43) (4,1.77)};		
		 \legend{constant,uniform,exponential};   
\end{axis}
\end{tikzpicture}
	\caption{Optimal price vectors for different inter arrival time distributions.}
	\label{fig:OptPriExp}
\end{figure}
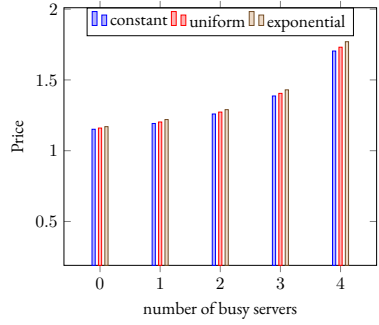

Next, we study the variation of the optimal revenue rate and optimal price with respect to arrival rate, service rate and number of servers. Consider a $5$ server system. The service rate $\mu=2$ and job valuations are distributed exponentially, with $\oG(p)=e^{-p}$. 
In Figure~\ref{fig:lamRev1}, we see that the optimal revenue increases monotonically as the arrival rate increases. This is expected, since a good pricing policy will be able to extract more revenue from increased demand. However, in Figure~\ref{fig:lamRev1}, we also see that the revenue per unit arrival rate is actually decreasing, as we scale up the arrival rate. This implies that the rate at which revenue can be extracted per unit arrival rate is decreasing. Both these observations validate the results of Proposition~\ref{prop:VarRevenueArrival}. 
In Figure~\ref{fig:lamRev2}, we have plotted the price vector for different arrival rates.  As the arrival rate increases, the price vector increases in all its components. 
Recall that this was conjectured in Remark~\ref{rem:LamVarRev}.
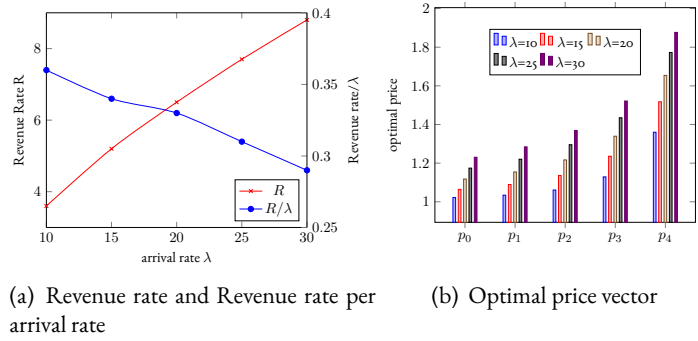
\begin{figure}
	\centering 
	\subfigure[Revenue rate and Revenue rate per arrival rate]
	{%
		\begin{tikzpicture}[scale=0.5]
		\pgfplotsset{
			xmin=10, xmax=30
		}
		
		\begin{axis}[
		axis y line*=left,
		ymin=3, ymax=9,
		xlabel=arrival rate $\lambda$,
		ylabel= Revenue Rate R,
		legend pos=south east
		]
		\addplot[smooth,mark=x,red] 
		coordinates{
			(10,3.6)(15,5.2)(20,6.5)(25,7.7)(30,8.8)
		}; \label{Hplot}
		\end{axis}
		
		\begin{axis}[
		axis y line*=right,
		axis x line=none,
		ymin=0.25, ymax=0.4,
		legend pos=south east,
		ylabel=Revenue rate/$\lambda$
		]
		\addlegendimage{/pgfplots/refstyle=Hplot}\addlegendentry{$R$}
		\addplot[smooth,mark=*,blue] 
		coordinates{
			(10,0.36)(15,0.34)(20,0.33)(25,0.31)(30,0.29)
		}; \addlegendentry{$R/\lambda$}
		\end{axis}	
\end{tikzpicture}
		\label{fig:lamRev1}
	}	
	\subfigure[Optimal price vector]
	{%
		\begin{tikzpicture}[scale=0.5]
		\begin{axis}[
		ybar,
		bar width=2pt,
		xlabel=\textcolor{white}{p},
		enlargelimits=0.15,
		legend style={at={(0.5,0.9)},
			anchor=north,legend columns=3},
		ylabel={optimal price},
		symbolic x coords={$p_0$, $p_1$, $p_2$, $p_3$, 
			$p_4$},
		xtick=data
		]
		\addplot+[ybar] plot coordinates {($p_0$,1.0220) ($p_1$,1.0343) 
			($p_2$,1.0609 ) ($p_3$,1.1290) ($p_4$,1.3599)};
		\addplot+[ybar] plot coordinates {($p_0$,1.0642) ($p_1$,1.0894) 
			($p_2$,1.1359 ) ($p_3$,1.2357) ($p_4$,1.5175)};
		\addplot+[ybar] plot coordinates {($p_0$,1.1178) ($p_1$,1.1543) 
			($p_2$,1.2168 ) ($p_3$,1.3394) ($p_4$,1.6541)};
		\addplot+[ybar] plot coordinates {($p_0$,1.1743) ($p_1$,1.2204) 
			($p_2$,1.2954) ($p_3$,1.4349) ($p_4$,1.7726)};		
		\addplot+[ybar] plot coordinates {($p_0$,1.2305) ($p_1$,1.2843) 
			($p_2$,1.3692) ($p_3$,1.5216) ($p_4$,1.8766)};
		\legend{\strut $\lambda$=10,\strut $\lambda$=15,\strut $\lambda$=20,\strut $\lambda$=25,\strut $\lambda$=30}
		\end{axis}
\end{tikzpicture}
		\label{fig:lamRev2}
	}
	\caption{Variation of optimal revenue rate, revenue rate per arrival rate and price  with arrival rate.}
\end{figure}

For studying the effect of service rate variation on optimal revenue rate, we again consider a $5$ server system, with arrival rate $\lambda=25$. As before, $\oG(p)=e^{-p}$.  In figure~\ref{fig:muRev1}, we see that revenue scales monotonically with service rate. Thus, by increasing the service capacity, we can extract more revenue. The revenue per service rate, however, decreases as service rate increases, in figure~\ref{fig:muRev1}. This implies that the marginal returns per unit service capacity decreases.  
These results are in line with Proposition~\ref{prop:VarRevenueService}. 
In Figure~\ref{fig:muRev2}, we see how the price vector decreases component wise as we increase the service rate, as expected in Remark~\ref{rem:MuVarRev}.

\begin{figure}
	\centering 
	\subfigure[Revenue rate and Revenue rate per service rate]
	{%
		\begin{tikzpicture}[scale=0.5]
		\pgfplotsset{
			xmin=1, xmax=5		
		}
		
		\begin{axis}[
		axis y line*=left,
		ymin=5, ymax=9,
		xlabel=arrival rate $\lambda$,
		ylabel= Revenue Rate R,
		legend pos=south west
		]
		\addplot[smooth,mark=x,red] 
		coordinates{
			(1,6)(2,7.7)(3,8.5)(4,8.8)(5,9)
		}; \label{Hplot1}
		\end{axis}
		
		\begin{axis}[
		axis y line*=right,
		axis x line=none,
		ymin=1.5, ymax=6,
		legend pos=south west,
		ylabel=Revenue rate/$\mu$
		]
		\addlegendimage{/pgfplots/refstyle=Hplot1}\addlegendentry{$R$}
		\addplot[smooth,mark=*,blue] 
		coordinates{
			(1,6)(2,3.9)(3,2.8)(4,2.2)(5,1.8)
		}; \addlegendentry{$R/\mu$}
		\end{axis}	
\end{tikzpicture}
		\label{fig:muRev1}
	}	
	\subfigure[Optimal price vector]
	{%
		\begin{tikzpicture}[scale=0.5]
		\begin{axis}[
		ybar,
		bar width=2pt,
		xlabel=\textcolor{white}{p}
		enlargelimits=0.15,
		legend style={at={(0.5,0.9)},
			anchor=north,legend columns=3},
		ylabel={optimal price},
		symbolic x coords={$p_0$, $p_1$, $p_2$, $p_3$, 
			$p_4$},
		xtick=data
		]
		\addplot+[ybar] plot coordinates {($p_0$,1.4309) ($p_1$,1.5052) 
			($p_2$,1.6155 ) ($p_3$,1.7997) ($p_4$,2.1960)};
		\addplot+[ybar] plot coordinates {($p_0$,1.1743) ($p_1$,1.2204) 
			($p_2$,1.2954) ($p_3$,1.4349) ($p_4$,1.7726)};
		\addplot+[ybar] plot coordinates {($p_0$,1.0813) ($p_1$,1.1105) 
			($p_2$,1.1628) ($p_3$,1.2710) ($p_4$,1.5652)};
		\addplot+[ybar] plot coordinates {($p_0$,1.0410) ($p_1$,1.0598) 
			($p_2$,1.0967) ($p_3$,1.1820) ($p_4$,1.4414)};
		\addplot+[ybar] plot coordinates {($p_0$,1.0220) ($p_1$,1.0343) 
			($p_2$,1.0609) ($p_3$,1.1290) ($p_4$,1.3599)};
		\legend{\strut $\mu$=1,\strut $\mu$=2,\strut $\mu$=3,\strut $\mu$=4,\strut $\mu$=5}
		\end{axis}
\end{tikzpicture}
		\label{fig:muRev2}
	}
	\caption{Variation of optimal revenue rate, revenue rate per service rate and price  with service rate}
\end{figure}
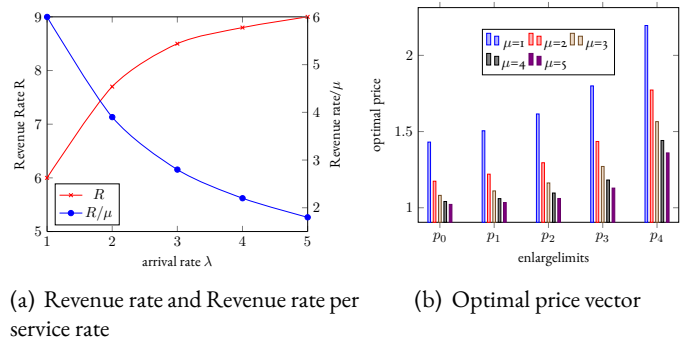	

For studying the relation between number of servers and optimal revenue/price, we consider a system with arrival rate $\lambda=25$, service rate $\mu=2$ and valuation distribution $\oG(p)=e^{-p}$. In figure~\ref{fig:servRev1}, we see how the optimal revenue and the optimal revenue per server vary, as we increase the number of servers. While we can extract more revenue as increase the number of servers, the revenue rate per server decreases.  We also see how the price vector itself behaves, as we increase the number of servers, in figure~\ref{fig:servRev2}. We see that the components of the price vector decrease and come closer to the optimal infinite server price $p^*_{\infty}$ (which equals 1 in this case), as we increase the number of servers. The trends are as predicted in Proposition~\ref{prop:VarRevenueServers}.

\begin{figure}
	\centering 
	\subfigure[Revenue rate and Revenue rate per server]
	{%
		\begin{tikzpicture}[scale=0.5]
		\pgfplotsset{
			xmin=3, xmax=7	
		}
		
		\begin{axis}[
		axis y line*=left,
		ymin=5.5, ymax=9,
		xlabel=arrival rate $\lambda$,
		ylabel= Revenue Rate R,
		legend pos=north east
		]
		\addplot[smooth,mark=x,red] 
		coordinates{
			(3,5.9)(4,6.9)(5,7.7)(6,8.3)(7,8.6)
		}; \label{Hplot2}
		\end{axis}
		
		\begin{axis}[
		axis y line*=right,
		axis x line=none,
		ymin=1.2, ymax=2,
		legend pos=north east,
		ylabel=Revenue rate/$\mu$
		]
		\addlegendimage{/pgfplots/refstyle=Hplot2}\addlegendentry{$R$}
		\addplot[smooth,mark=*,blue] 
		coordinates{
			(3,2)(4,1.7)(5,1.5)(6,1.4)(7,1.2)
		}; \addlegendentry{$R/\mu$}
		\end{axis}	
\end{tikzpicture}
		\label{fig:servRev1}
	}	
	\subfigure[Optimal price vector]
	{%
		\begin{tikzpicture}[scale=0.5]
		\begin{axis}[
		ybar,
		bar width=2pt,
		xlabel=\textcolor{white}{p},
		enlargelimits=0.15,
		legend style={at={(0.5,1.1)},
			anchor=north,legend columns=3},
		ylabel={optimal price},
		symbolic x coords={$p_0$, $p_1$, $p_2$, $p_3$, 
			$p_4$, $p_5$, $p_6$},
		xtick=data,
		extra y ticks = 1,
		extra y tick labels={},
		extra y tick style={grid=major,major grid style={draw=red}}
		]
		\addplot+[ybar] plot coordinates {($p_0$,1.4459) ($p_1$,1.6102) 
			($p_2$,1.9813) ($p_3$,0) ($p_4$,0) ($p_5$,0) ($p_6$,0)};
		\addplot+[ybar] plot coordinates {($p_0$,1.2801) ($p_1$,1.3642) 
			($p_2$,1.5153) ($p_3$,1.8687) ($p_4$,0) ($p_5$,0) ($p_6$,0)};
		\addplot+[ybar] plot coordinates {($p_0$,1.1743) ($p_1$,1.2204) 
			($p_2$,1.2954) ($p_3$,1.4349 ) ($p_4$,1.7726) ($p_5$,0) ($p_6$,0)};
		\addplot+[ybar] plot coordinates {($p_0$,1.1053) ($p_1$,1.1312) 
			($p_2$,1.1709) ($p_3$,1.2376 ) ($p_4$,1.3663) ($p_5$,1.6898) ($p_6$,0)};
		\addplot+[ybar] plot coordinates {($p_0$,1.0608) ($p_1$,1.0752) 
			($p_2$,1.0964) ($p_3$,1.1302 ) ($p_4$,1.1893 ) ($p_5$,1.3079) ($p_6$,1.6180)};
		\legend{\strut K=3, \strut K=4, \strut K=5, \strut K=6, \strut K=7}
		\addlegendimage{my legend}
		\addlegendentry{\strut $p^*$}
		\end{axis}
\end{tikzpicture}
		\label{fig:servRev2}	
	}
	\caption{Variation of optimal revenue rate, revenue rate per service rate and price  with number of servers}
\end{figure}
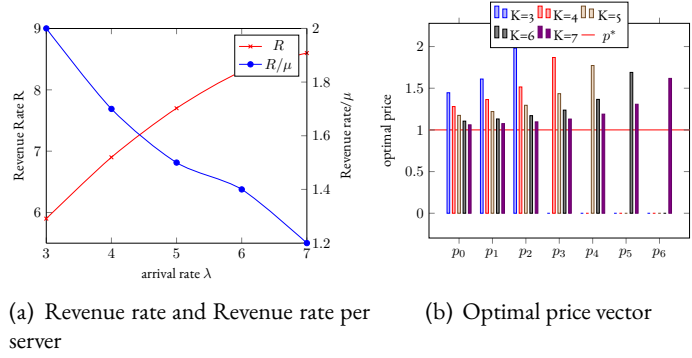

We compare differential pricing and uniform pricing for a system with Poisson arrivals (or equivalently, exponential inter arrival time). 
We consider a 5-server system, with $\mu=2$. 
For different values of load $\rho=\dfrac{\lambda}{\mu}$, we compare the revenue under the optimal price $\bp^*$ with the revenue under uniform prices $p^*_{\infty}$ and  $p^*_5$. 
The valuation function $\oG(p)=e^{-p}$. The resultant values are displayed in Figure~\ref{fig:fiveservers}.
\begin{figure}
	\centering 
	\subfigure[Revenue vs Load, 5 servers]
	{\label{fig:fiveservers}
		\begin{tikzpicture}[scale=0.7]
  \centering
  \begin{axis}[
        ybar, axis on top,
        height=5cm, width=6cm,
        bar width=0.2cm,
        ymajorgrids, tick align=inside,
        enlarge y limits={value=.1,upper},
        ymin=0, ymax=7,
        y axis line style={opacity=100},
        tickwidth=0pt,
        enlarge x limits=true,
        legend style={
            at={(0.4,1)},
            anchor=north,
            legend columns=-1,
            /tikz/every even column/.append style={column sep=0.5cm}
        },
        ylabel={Revenue Rate},
        symbolic x coords={
           0.5,1,5,10},
       xtick=data,
       xlabel=load $\rho$
    ]
    \addplot [draw=black, fill=blue!30] coordinates {
      (0.5, 0.37)
      (1, 0.73) 
      (5,3.57) 
      (10,6.10)}; 
      \addplot [draw=black, fill=red!30] coordinates {
      (0.5, 0.37)
      (1, 0.73) 
      (5,3.58) 
      (10,6.46)};
      \addplot [draw=black, fill=green!30] coordinates {
      (0.5, 0.37)
      (1, 0.73) 
      (5,3.59) 
      (10,6.54)};
    \legend{ $p^*_{\infty}$, $p_5^*$,$\bp^*$}
  \end{axis}
  \end{tikzpicture}
	}	
	\subfigure[Revenue vs Load, 10 servers]
	{\label{fig:tenservers}
		\begin{tikzpicture}[scale=0.7]
  \centering
  \begin{axis}[
        ybar, axis on top,
        height=5cm, width=6cm,
        bar width=0.2cm,
        ymajorgrids, tick align=inside,
        enlarge y limits={value=.1,upper},
        ymin=0, ymax=17,
        y axis line style={opacity=100},
        tickwidth=0pt,
        enlarge x limits=true,
        legend style={
            at={(0.4,1)},
            anchor=north,
            legend columns=-1,
            /tikz/every even column/.append style={column sep=0.5cm}
        },
        ylabel={Revenue Rate},
        symbolic x coords={
           5,10,15,25},
       xtick=data,
       xlabel=load $\rho$
    ]
    \addplot [draw=black, fill=blue!30] coordinates {
      (5, 3.67) 
      (10,7.33) 
      (15,10.7) 
      (25,15.13)}; 
      \addplot [draw=black, fill=red!30] coordinates {
      (5, 3.67) 
      (10,7.33) 
      (15,10.78) 
      (25,16.46)};
      \addplot [draw=black, fill=green!30] coordinates {
      (5, 3.67) 
      (10,7.34)
      (15,10.83) 
      (25,16.69)};
    \legend{$p^*_{\infty}$, $p_5^*$,$\bp^*$}
  \end{axis}
  \end{tikzpicture}
	}
	\subfigure[Optimal revenue vs number of servers]
	{\label{fig:revVariaServ}
		\begin{tikzpicture}[scale=0.5]
\begin{axis}[
    xlabel={Number of servers},
    ylabel={Revenue rate},
    xmin=4, xmax=10,
    ymin=6, ymax=8,
    xtick={5,6,7,8,9,10},
    ytick={6,6.2,6.4,6.6,6.8,7.0,7.2,7.4,7.6,7.8,8.0},
    legend pos=south east,
    ymajorgrids=true,
    grid style=dashed,
]
 
\addplot[
    color=blue,
    mark=square,
    ]
    coordinates {
    (5,6.54)(6,6.90)(7,7.12)(8,7.24)(9,7.31)(10,7.34)
    };
    \addplot[
    color=red,
    ]
    coordinates {
    (4,7.36)(5,7.36)(6,7.36)(7,7.36)(8,7.36)(9,7.36)(10,7.36)
    };
    \legend{optimal revenue, infinite server revenue}
    
\end{axis}
\end{tikzpicture}
	}
	\caption{Variation of optimal revenue and price  with number of servers}
\end{figure}
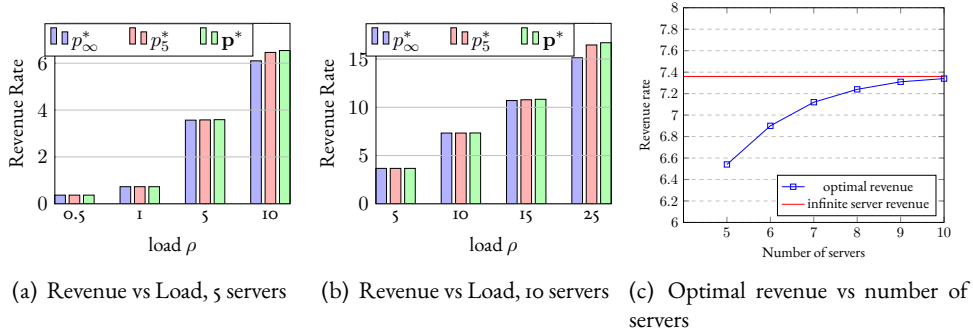
At low values of arrival rates, differential pricing does not offer substantial gains over uniform pricing. 
At higher arrival rates, however, we begin to see that revenue rates show a significant improvement using differential pricing. 
One can also see that these effects are more pronounced beyond $\rho=5$, the number of servers. 
A similar effect is seen in the case of 10 servers as well, as seen in Figure~\ref{fig:tenservers} (all other parameters remaining same). 
Beyond $\rho=10$, differential pricing begins to outperform uniform pricing.

In the Fig.~\ref{fig:revVariaServ}, we also study how quickly the optimal differential revenue for a finite server system converges to the optimal revenue with infinite servers, for a system with Poisson arrivals. 
We fix $\lambda=1$ and $\mu=2$ and $\oG(p)=e^{-p}$. It is clear that the infinite server optimal revenue, $R(\infty,p^*_{\infty}\bI)=7.36$. With as few as 10 servers, we come close to the infinite server revenue. With these many servers, optimal pricing can be closely approximated by optimal uniform pricing. Recall from Remark~\ref{rem:UnifBoundPoiss} that with $K=10$, the optimal revenue can exceed the optimal uniform revenue by at most $5\%$.

\section{Conclusion}
We studied optimal service pricing in server farms where customers arrive according to a renewal process and have \iid exponential service times and \iid valuations of the service. 
We showed that fixed pricing achieves optimal revenue rate in infinite server systems but can guarantee only close to optimal revenue rate in finite server systems. 
However, fixed pricing suffices to drive revenue rate to infinity in infinite server systems as the arrival rates increase. 
We also showed that the optimal prices for finite server systems increase with the number of busy servers. In case of exponential interarrival  times, we derived several properties of the optimal prices vis a vis arrival rates, service rates, and the number of servers in the system. 

We argued that for a given service rate, the optimal revenue rates in server farms with non-exponential service times generally exceed those in server farms with exponential service times. But, the optimal prices in the former systems depend on the elapsed service times in the busy servers and cannot be obtained using the Markov control framework as we have done in this work. Similarly, optimal pricing for processor sharing systems and systems with queues where customers can wait for service are also challenging problems. These are potential topics for future research.  

\bibliographystyle{plain}
\bibliography{optimal}

\appendix
\section{Examples for asymptotic revenue rates}
\label{sec:Examples}
\subsection{Fixed uniform pricing}
\label{subsec:FixUnifPrice}
We first present an example where the limiting revenue rate is $\mu pK$.
\begin{exmp}
	\label{exmp:Max}
	Consider the Poisson arrival process for jobs, with rate $\lambda$. 
	Then, the interarrival times are exponential, 
	and the Laplace Steiltjes transform $\phi(\mu)=\frac{\lambda}{\lambda+\mu}$.  
	We can write the limit
	$
	\lim_{\lambda\to\infty}\lambda(1-\phi(\mu))=\lim_{\lambda\to\infty}\mu\frac{\lambda}{\lambda+\mu} = \mu.
	$
	It follows that $\tilde\mu = \mu$ in Theroem~\ref{thm:AsympRevLST}, 
	and hence the limiting mean revenue rate is  $\lim_{\lambda\to\infty}R(K,p\bI) = \mu Kp$ for $K$ server system under uniform price $p\bI$.  
\end{exmp}

We next present an example of an interarrival distribution for which the limiting revenue rate goes to $0$. 
\begin{exmp} 
	\label{exmp:Zero}
	For some $m > 1$, 
	we consider the job interarrival times $(U_n \in \set{\sqrt{m}, \frac{1}{m}}: n \in \N)$ such that $P\set{U_1 = \sqrt{m}} = 1/m$. 	In this case the arrival rate $\lambda = 1/\E U_1$ where the mean interarrival time $\E U_1 = \frac{1}{\sqrt{m}}+ \frac{1}{m}(1-\frac{1}{m})$.  
	It follows that $m \to \infty$ implies $\lambda \to \infty$. 
	We next compute the Laplace Stieltjes transform $\phi(x) = \E e^{-xU_1}$ of interarrival times as 
	\begin{equation*}
	\phi(x)=(1-\frac{1}{m})e^{-\frac{x}{m}}+\frac{e^{-x\sqrt{m}}}{m}.
	\end{equation*}
	Using the fact that $\lambda= 1/\E U_1$, we can write the limit  
	\begin{equation*}
	\lim_{m \to \infty}\lambda(1-\phi(x))
	=\lim_{m \to \infty}\frac{\frac{m^2}{m-1}-me^{-\frac{x}{m}}-\frac{m}{m-1}e^{-x\sqrt m}}{1+\frac{m\sqrt m}{m-1}}
	= 0.
	\end{equation*}
	For this interarrival distribution, 
	it follows from Theorem~\ref{thm:AsympRevLST} that the limiting mean revenue rate is zero 
	for any uniform price $p$. 
	One would expect the mean revenue rate to increase with the arrival rate, 
	since the mean interarrival time decreases. 
	However, for this example distribution, 
	the mean revenue rate instead of increasing with the arrival rate, goes to zero.  
	Such a behavior arises due to slow decay of the tail of the interarrival time distribution. 
\end{exmp}

There are distributions for which the limiting revenue is non zero but strictly less than $\mu pK$, 
as in the following example. 
\begin{exmp}
	\label{exmp:Inter}
	For $m > 1$, 
	consider the \iid job interarrival times $(U_n\in\set{1, \frac{1}{m}}: n \in \N)$ with $P\set{U_1 = 1} = \frac{1}{m}$. 
	That is, the mean interarrival time $1/\lambda = \E U_1 = \frac{2m-1}{m^2}$ and the Laplace Stieltjes transform $\phi(x) = \E e^{-xU_1} = \frac{1}{m}e^{-x} + (1-\frac{1}{m})e^{-\frac{x}{m}}$ for $x \in \R_+$. 
	It follows that $\lambda\to\infty$ as $m$ grows large, 
	and the limit 
	\begin{equation*}
	\tilde\mu = \lim_{m\to\infty}\lambda(1-\phi(\mu)) 
	= \frac{1-e^{-\mu}}{2}. 
	\end{equation*} 
	It follows from Theorem~\ref{thm:AsympRevLST}  
	that the limiting revenue rate is $\lim_{\lambda\to\infty}R(K,p\bI) = \dfrac{1-e^{-\mu}}{2}pK$. 
	Since $1-e^{-\mu}\le \mu$, the limiting revenue rate is smaller than $\dfrac{\mu pK}{2}$.
\end{exmp}

\subsection{Arrival rate dependent uniform pricing}
\label{subsec:ArrDepUnifPrice}
\begin{exmp}
	\label{exmp:Linear}
	If the arrival process is Poisson and the value distribution is Pareto, i.e. $\oG(x)=\frac{\theta}{x}\SetIn{x \ge \theta}$, 
	then the choice of uniform price $p(\lambda) = \oG^{-1}(\frac{1}{\lambda})= \lambda\theta$ that grows linearly with the arrival rate $\lambda$. 
	Further, we have the Laplace Stieltjes transform or \iid interarrival times $\phi(\mu) = \frac{\lambda}{\lambda+\mu}$, 
	and hence $\tilde{\nu} = \lim_{\lambda\to\infty}\lambda(1-\phi(\mu)) = \mu > 0$. 
	Thus results in the mean revenue rate $R(K, p\bI)$ is asymptotically linearly increasing in the arrival rate $\lambda$.  
\end{exmp}
\begin{exmp} 
	\label{exmp:Log}
	Let the arrival process be Poisson with rate $\lambda$ and the complimentary value distribution $\oG(x)=c_1e^{-c_2x^2}$. Since the arrival process is Poisson, $\tilde\mu=\mu$. Choosing the uniform price $p(\lambda) = \sqrt{\frac{1}{c_2}\log(c_1\lambda)}$, 
	we see that $\lim_{\lambda \to \infty}R(K, p\bI) = \infty$. 
	Contrastingly, for a uniform price $p(\lambda) = \log\lambda$, 
	we get 
	$\lim_{\lambda \to \infty}R(K, p\bI) = 0$. 
\end{exmp}

\section{Proofs for MDP with Poisson arrival process}
\label{sec:PoissonArrMDP}

\subsection{Proof of Lemma~\ref{lem-m-monotonicity}}
\label{sec:ProofMMonotone}
Let $m$ and $u^\ast$ be as defined in Eq.~\eqref{eqn:def-m} and Eq.~\eqref{eqn:DefnArgMax} respectively.  
For $B_1,B_2 \in \R$, we let $u_i \in u^\ast(B_i)$ for $i \in \set{1,2}$. 
From the definition of $f$, we have $f(B_2, u_2) - f(B_1,u_2) = -(B_2-B_1)\oG(u_2)$.   
In addition, we have $f(B_1,u_1) \ge f(B_1,u_2)$ from the definition of $m, u_1, u_2$.  
Therefore, we can lower bound the difference  
\EQN{
	\label{eqn:mDiffLB}
	m(B_1) - m(B_2) 
	\ge (B_2-B_1)\oG(u_2). 
} 
Similarly, we have $f(B_1,u_1)-f(B_2,u_1)=(B_2-B_1)\oG(u_1)$ and $f(B_2,u_2) \ge f(B_2,u_1)$. 
Therefore, we can upper bound the difference
\EQN{
	\label{eqn:mDiffUB}
	m(B_1) - m(B_2) 
	\le (B_2-B_1)\oG(u_1). 
} 

\begin{enumerate}[(a)]
	\item 
	Since $f(B,B) = 0$, 
	it follows that $m(B) \ge 0$ for all $B$. 
	Since $\oG \ge 0$, it follows that $m(B_2)-m(B_1) \le 0$ for $B_1 < B_2$ from Eq.~\eqref{eqn:mDiffLB}. 
	\item 
	Since $\oG \le 1$, it follows from Eq.~\eqref{eqn:mDiffUB} that $0 \le m(B_1)-m(B_2) \le B_2-B_1$.  
	Similarly for $B_2 < B_1$, we observe that $0 \le m(B_2)- m(B_1) \le B_1-B_2$. 
	Combining these, we have  $\abs{m(B_1)-m(B_2)} \le \abs{B_1-B_2}$, 
	implying Lipschitz-1 continuity of $m$.  
	Finally, $f(B, u)$ is affine, and hence, convex in $B$. 
	Hence, the maximum $m$ of convex functions $f(B,u)$ is also convex in $B$~\cite[Section 3.2.3]{boyd2004convex} .
	\item 
	Subtracting Eq.~\eqref{eqn:mDiffLB} from Eq.~\eqref{eqn:mDiffUB}, we get $(B_2-B_1)(\oG(u_2) - \oG(u_1)) \le 0$. 
	This implies that if $B_2 > B_1$, then $\oG(u_2) \le \oG(u_1)$. 
	The monotonic decrease of $\oG$ implies that $u_2 \ge u_1$.
\end{enumerate}

\subsection{Proof of Lemma~\ref{lem-delta-monotonicity}}
\label{sec:ProofDeltaMonotone}
We assume the Lemma hypothesis. 
\begin{enumerate}[(a)]
	\item Using Eq.~\eqref{eqn:iterative} for $i=0$, we see that $\theta=\lambda m(\Delta(0))$. The result follows from the  non-negativity of $m$ from Lemma~\ref{lem-m-monotonicity}. 
	\item We first prove that $\Delta(0) > 0$ via contradiction. 
	Assume that $\Delta(0) \le 0$, 
	and assume the inductive hypothesis that $\Delta(i) \le 0$ for some $i \in \cX'\setminus\set{0}$. 
	Then, it follows from Eq.~\eqref{eqn:iterative}  
	\EQ{
		m(\Delta(i)) 
		= \frac{\theta - i\mu \Delta(i-1)}{\lambda} \ge \frac{\theta}{\lambda} 
		=  m(\Delta(0)) \ge 0. 
	}
	From monotone decrease of $m$ in Lemma~\ref{lem-m-monotonicity}(a) and the induction step, 
	it follows that $\Delta(i)\le \Delta (0) \le 0$ for all $i \in \cX'$. 
	From Eq.~\eqref{eqn:iterative} for $i=K$,  we get $\Delta(K-1) = \frac{\theta}{K\mu}\ge 0$ from the non negativity of $\theta$ from part (a). 
	This leads to a contradiction and hence we see that $\Delta(0) > 0$. 
	
	From Eq.~\eqref{eqn:iterative} for $i=1$ and positivity of $\Delta(0)$, we observe that 
	\EQ{
		m(\Delta(1)) = \frac{\theta - \mu\Delta(0)}{\lambda} \le \frac{\theta}{\lambda} = m(\Delta(0)). 
	}
	Lemma~\ref{lem-m-monotonicity}(a) implies that  $m$ is decreasing. It follows that $\Delta(1) \ge \Delta(0)$. 
	Assuming the inductive hypothesis $\Delta(i-1) \ge \Delta(i-2)$ for some $i \in \set{2, \dots, K-1}$ and positivity of $\Delta(i)$s, 
	we get from Eq.~\eqref{eqn:iterative}
	\EQ{
		m(\Delta(i-1))-m(\Delta(i)) = \frac{\mu}{\lambda}(i\Delta(i-1)-(i-1)\Delta(i-2)) \ge 0.
	}
	From monotone decrease of $m$ and the induction step, 
	it follows that $\Delta(i) \ge \Delta(i-1)$ for all $i \in \cX'\setminus\set{0}$. 
	\item This follows by combining the monotone increase of $\Delta(i)$ shown in part~(b), 
	and monotonicity of $u^\ast(B)$ in $B$ shown in Lemma~\ref{lem-m-monotonicity}(c).  
\end{enumerate}

\subsection{Proof of Proposition~\ref{prop:VarRevenueArrival}}
\label{sec:ProofRevenueArr}
Notice that the optimal revenue rate $\theta^\ast(\lambda)$ is the solution to Eqs.~\eqref{eqn:iterative2a}-\eqref{eqn:iterative2c} 
as a function of $\lambda$, for a fixed $\mu$ and $K$. 
\begin{enumerate}[(a)]
	\item To begin with let us fix both $\theta$ and $\mu$ and vary $\lambda$ in Eqs.~\eqref{eqn:iterative2b} and~\eqref{eqn:iterative2c}. 
	It follows that if $g_i$ is non-increasing in $\lambda$, 
	then $m(g_i)$ is non-decreasing in $\lambda$ from its monotone decrease property. 
	Since $g_{i-1} \propto \theta - \lambda m(g_i)$, 
	it follows that $g_{i-1}$ is decreasing  and $m(g_{i-1})$ is increasing in $\lambda$. 
	Since $g_{K-1} = \theta/K\mu$ is constant in $\lambda$, 
	it follows that $m(g_i)$ is increasing in $\lambda$ for all $i \in \cX'$ and fixed $\theta$ and $\mu$. 
	Since  $\theta^\ast(\lambda) = \lambda m(g_0)$ from Eq.~\eqref{eqn:iterative2a}, 
	it follows that the optimal revenue rate $\theta^\ast(\lambda)$ is increasing in $\lambda$ 
	for a fixed $\mu$.     
	\item 
	The argument is via contradiction. 
	Let $\theta^\ast(\lambda)/\lambda$ increase with $\lambda$. 
	Observe that $g_{K-1}(\theta^\ast(\lambda)) = \frac{\theta^\ast(\lambda)}{K\mu}$ increases with $\lambda$. 
	Since $\theta^\ast$ is the solution to Eqs.~\eqref{eqn:iterative2a}-\eqref{eqn:iterative2c} for all $i \in \cX$, 
	\EQ{
		\frac{g_{i-1}(\theta^\ast(\lambda))}{\lambda} = \frac{\theta^\ast(\lambda)/\lambda - m(g_i(\theta^\ast(\lambda)))}{i\mu},\quad i \in [K-1]. 
	}
	It follows that $g_{i-1}(\theta^\ast(\lambda))/\lambda$ is an increasing function of $\lambda$, 
	if $g_i$ is an increasing function of $\lambda$. 
	It follows from induction that $g_{0}(\theta^\ast(\lambda))$ is an increasing function of $\lambda$, 
	and hence $m(g_0(\theta^\ast(\lambda)) = \theta^\ast(\lambda)/\lambda$ is a decreasing function of $\lambda$. 
	This leads to a contradiction. 
\end{enumerate}

\subsection{Proof of Proposition~\ref{prop:VarRevenueService}}
\label{sec:ProofRevenueService}

The optimal revenue rate $\theta^\ast(\mu)$ is the solution to Eqs.~\eqref{eqn:iterative2a}-\eqref{eqn:iterative2c} 
as a function of $\mu$, for a fixed $\lambda$ and $K$.
\begin{enumerate}[(a)] 
	\item  To begin with let us fix both $\theta$ and $\mu$ and vary $\lambda$ in Eqs.~\eqref{eqn:iterative2b} and~\eqref{eqn:iterative2c}. From Eq.~\eqref{eqn:iterative2b}, 
	we observe that $g_{i-1} = (\theta - \lambda m(g_i))/i\mu$ for $i \in [K-1]$. 
	Hence, if $g_i$ is decreasing with $\mu$, then $m(g_i)$ is increasing in $\mu$ due to its monotone decrease property, 
	and hence $g_{i-1}$ is decreasing with $\mu$. 
	Since $g_{K-1} = \theta/K\mu$ from Eq.~\eqref{eqn:iterative2c} for $i = K$, 
	it follows by induction that $g_0$ is decreasing and hence $\lambda m(g_0)$ is increasing in $\mu$.  
	As a result, if we increase $\mu$ keeping $\lambda$ fixed, 
	the average revenue rate $\theta^\ast(\mu)$, 
	the solution to $\theta = \lambda m(g_0(\theta))$ increases in $\mu$.  
	\item The argument is via contradiction. 
	Let $\theta^\ast(\mu)/\mu$ increase with $\mu$. 
	We obtain from Eq.~\eqref{eqn:iterative2b} for $i \in [K-1]$, 
	\EQ{
		g_i(\theta^\ast(\mu)) = m^{-1}\left(i\mu\left(\frac{\theta^\ast(\mu)/i\mu - g_{i-1}(\theta^\ast(\mu))}{\lambda}\right)\right).
	}
	Then, it follows that if $g_{i-1}$ is decreasing with $\mu$, 
	then $g_i$ is also decreasing in $\mu$. 
	From Eq.~\eqref{eqn:iterative2a} for $i = 0$, we see that $g_0(\theta^\ast(\mu)) = m^{-1}(\frac{\theta^\ast(\mu)}{\lambda})$ is decreasing with $\mu$, 
	and hence it follows that $g_{K-1}$ is decreasing and in $\mu$. 
	However $g_{K-1}(\theta^\ast) = \theta^\ast(\mu)/K\mu$ was assumed to be increasing in $\mu$, 
	that leads to a contradiction. 
\end{enumerate}

\subsection{Proof of Proposition~\ref{prop:VarRevenueServers}}
\label{sec:ProofRevenueServers}
We define functions 
\begin{align*}
\bg_0(\theta) &\triangleq m^{-1}\left(\frac{\theta}{\lambda}\right),\\
\bg_i(\theta) &\triangleq m^{-1}\left(\frac{\theta - i\mu \bg_{i-1}(\theta))}{\lambda}\right),~i \in [K-1].
\end{align*}
Following similar arguments as in the proof of Theorem~\ref{thm:fixed-point-eqn}(a) we can iteratively show that 
$\bg_i(\theta)$ are decreasing in $\theta$ for all $i < K$. 
\begin{enumerate}[(a)]
	\item It follows that $\lambda m(\bg_0) = \theta$.  
	and $\bg_{i-1} = (\theta - \lambda m(\bg_i))/i\mu$ for $i \in [K-1]$. Hence, from Eqs.~\eqref{eqn:iterative2a}-\eqref{eqn:iterative2c} it follows that the optimal average reward $\theta^\ast(K)$ is the solution to the fixed point equation 
	$\theta = K\mu \bg_{K-1}(\theta))$. 
	From Lemma~\ref{lem-delta-monotonicity}(b), 
	we have $\bg_{i}(\theta) > \bg_{i-1}(\theta)$ for all $\theta \ge 0$ and $i \in \cX^\prime$. In particular, $(K+1)\mu\bg_K(\theta)>K\mu\bg_{K-1}(\theta)$ for all $\theta\ge 0$.
	Hence we can infer that $\theta^\ast(K)$ increases with $K$. 
	\item 
	Since $\bg_{K-1}(\theta^\ast(K)) = \theta^\ast(K)/K\mu$, 
	it suffices to show that $\bg_{K-1}(\theta^\ast(K))$ is decreasing in $K$. 
	We show this by contradiction. 
	To this end, we assume that $\bg_{K}(\theta^\ast(K+1)) > \bg_{K-1}(\theta^\ast(K))$. 
	Together with this hypothesis and monotone increase of $\bg_i$ from Lemma~\ref{lem-delta-monotonicity}(b), 
	we obtain 
	\EQ{
		(K+1)\bg_{K}(\theta^\ast(K+1)) - K\bg_{K-1}(\theta^\ast(K)) > \bg_{0}(\theta^\ast(K+1)).
	}
	Multiplying both the sides by $\mu/\lambda$ and using definitions of $\theta^\ast(K)$ and $\theta^\ast(K+1)$, the above inquality 
	reduces to
	\EQ{
		\frac{\theta^\ast(K+1) - \mu\bg_{0}(\theta^\ast(K+1))}{\lambda} > \frac{\theta^\ast(K)}{\lambda}. 
	}
	From the monotone decrease property of $m$ and definition of $\bg_1$ and $\bg_0$, 
	we obtain
	$\bg_1(\theta^\ast(K+1)) < \bg_0(\theta^\ast(K))$. 
	We will inductively show that $\bg_i(\theta^\ast(K+1)) < \bg_{i-1}(\theta^\ast(K))$ for all $i \in [K]$.
	We have already shown the base case of $i = 1$. 
	We assume that the inductive hypothesis holds for some $i \in [K-1]$. 
	Further, Lemma~\ref{lem-delta-monotonicity}(b) implies that $\bg_i$ increases in $i$ for a fixed argument. 
	Together with inductive and initial hypothesis, we obtain 
	\eq{
		K(&\bg_K(\theta^\ast(K+1)) - \bg_{K-1}(\theta^\ast(K))) \\ 
		+ &(\bg_K(\theta^\ast(K+1)) - \bg_i(\theta^\ast(K+1)))\\
		& > 0 >  i(\bg_i(\theta^\ast(K+1)) - \bg_{i-1}(\theta^\ast(K))).
	}
	Rearranging the terms, 
	multiplying both the sides by $\mu/\lambda$, 
	using definitions of $\theta^\ast(K), \theta^\ast(K+1),\bg_i, \bg_{i+1}$, 
	and from the monotone decrease of $m$, 
	we get 
	\EQ{
		\bg_{i+1}(\theta^\ast(K+1)) < \bg_{i}(\theta^\ast(K)).
	}
	This completes the induction step. We thus see that $\bg_i(\theta^\ast(K+1)) < \bg_{i-1}(\theta^\ast(K))$ for all $i \in [K]$.
	In particular, we get $\bg_K(\theta^\ast(K+1)) < \bg_{K-1}(\theta^\ast(K))$
	which contradicts the initial hypothesis.
	
	\item 
	Recall that the optimal price for $i$ busy servers, 
	when the system has $K$ servers is given by 
	\EQ{
		u_i^\ast(K) = u^\ast(\bg_i(\theta^\ast(K))). 
	}
	We know that $\theta^\ast(K)$ is increasing in $K$ from part (a) of the proof, 
	$\bg_i(\theta)$ is decreasing in $\theta$ as observed in the beginning of the proof, 
	and  $u^\ast$ is non-decreasing in its argument from Lemma~\ref{lem-m-monotonicity}(c). 
	The result follows from the combination of these three observations. 
	%
\end{enumerate}

\section{Proofs for MDP with general arrival process}
\label{sec:GenArrMDP}

\begin{lem}
	\label{lem:ProbDepartProps}
	For the probability $\alpha_{k,j}$ of $j$ departures from state $k$ defined in Eq.~\eqref{eqn:NumDeparturesSampledProcess}, 
	for the $K$ server system with \iid exponential service wih rate $\mu$ and \iid job interarrival times $(U_n: n \in\N)$, 
	the following statements are true for all $i \in \cX^\prime$. 
	\begin{enumerate}[(a)]
		\item $a_{i+1,j} \le a_{i,j}$ for all $j \le i-1$. 
		\item $a_{i,i-1}+\alpha_{i,0} =1$.
		\item $\sum_{j=0}^ia_{i+1,j}-\sum_{j=0}^{i-1}a_{i,j} = \alpha_{1,1}$.
	\end{enumerate}
\end{lem}
\begin{proof}
	Given the first interarrival time $U_1$, 
	we can define a sequence of conditionally \iid Bernoulli random variables $(\xi_r: r \in \N)$ such that $\E[\xi_r|U_1] = 1-e^{-\mu U_1}$. 
	We define a sequence of increasing binomial random variables $(X_k: k \in \N)$ such that $X_k \triangleq \sum_{r\in [k]}\xi_r$. 
	We observe that 
	\EQ{
		\E[\SetIn{X_k = k-i}| U_1] = \binom{k}{k-i}(1-e^{-\mu U_1})^{k-i}e^{-i\mu U_1}.
	} 
	Therefore, it follows from Eq.~\eqref{eqn:NumDeparturesSampledProcess} that the probability of $k-i$ departures from state $k$ is $\alpha_{k,k-i} = P\set{X_k = k-i}$. 
	\begin{enumerate}[(a)]
		\item Recall that $a_{i,j} = \sum_{l=0}^j\alpha_{i,i-l}$ and hence we can write $a_{i,j} = P\set{X_i \le j}$. 
		Since $X$ is monotonically increasing $\set{X_{k+1} \le j} \subseteq \set{X_k \le j}$, 
		and the result follows from the monotonicity of probability. 
		\item From the definition of random sequence $X$ and $a_{i,j}= \sum_{l=0}^j\alpha_{i,i-l}$, 
		we have $a_{i,i} = P\set{X_i \le i} = 1 = \alpha_{i,0}+a_{i,i-1}$.
		\item From the definition of $(a_{ij}: j \le i)$ and monotone sequence $X$, 
		we can write
		\EQ{ 
			\sum_{j=0}^ia_{i+1,j} 
			= \sum_{l=1}^{i+1}l P\set{X_{i+1}=l}
			= \E X_{i+1}.
		}
		Since $X_{i+1}-X_i = \xi_{i+1}$ and $\E\xi_{i+1} = \E(1-e^{-\mu U_1}) = \alpha_{1,1}$, the result follows.  
	\end{enumerate}
\end{proof}

\subsection{Proof of Theorem~\ref{thm:DeltaMonotonicityG}}
\label{subsec:GenArrMDPThm}

Recall that the map $m$ is non-negative and monotonically decreasing from Lemma~\ref{lem-m-monotonicity}(a), 
and that $m$ is Lipschitz-1 continuous from Lemma~\ref{lem-m-monotonicity}(b). 
\begin{enumerate}[(a)]
	\item From the non-negativity of $m$ and writing $\theta = \lambda m(b_0)$ from Eq.~\eqref{eq:bellman-gm1}, 
	we get the result. 
	\item We first prove that $\Delta(0) > 0$ via contradiction. 
	Assume that $\Delta(0) \le 0$. 
	We will show by induction that for all $i \in \set{2, \dots, K}$, 
	the following two conditions hold true,  
	\begin{subequations}
		\begin{align}
		\label{eqn:InductiveHypothesis1}
		&\Delta(i-1) < \dots < \Delta(0) \le 0,\\
		\label{eqn:InductiveHypothesis2}
		&\sum_{j=0}^{i-1}a_{i,j}\Delta(j) < \sum_{j=0}^{i-2}a_{i-1,j}\Delta(j) \le 0.
		\end{align}
	\end{subequations} 
	The inductive hypothesis in Eqs.~\eqref{eqn:InductiveHypothesis1}-\eqref{eqn:InductiveHypothesis2} at $i=K$, 
	together with equality in Eq.~\eqref{eq:bellman-gm3}, 
	we get $\frac{\theta}{\lambda} = \sum_{j=0}^{K-1}a_{K,j}\Delta(j) < 0$. 
	This contradicts the part (a) of the theorem, which we have already established. 
	The contradiction implies that following two conditions hold true for all $i \in [K]$, 
	\begin{subequations}
		\begin{align}
		\label{eqn:Implication1}
		&\Delta(i-1) \ge \dots \ge \Delta(0) > 0,\\
		\label{eqn:Implication2}
		&\sum_{j=0}^{i-1}a_{i,j}\Delta(j) \ge \sum_{j=0}^{i-2}a_{i-1,j}\Delta(j) > 0.
		\end{align}
	\end{subequations} 
	From Eqs.~\eqref{eq:bellman-gm1}-\eqref{eq:bellman-gm2} and the condition~\eqref{eqn:Implication2}, 
	we observe that $m(b_i) < m(b_{i-1})$. 
	From the monotonicity of map $m$ in Lemma~\ref{lem-m-monotonicity}, 
	we observe that the sequence $b$ is non-decreasing. 
	Further, since $b_0 = m^{-1}(\frac{\theta}{\lambda}) \ge 0$, 
	and the result follows.  
	Therefore, it suffices to show the inductive hypothesis in Eqs.~\eqref{eqn:InductiveHypothesis1},~\eqref{eqn:InductiveHypothesis1} holds true for all $i \ge 2$. 
	
	\textbf{Step 1: Base case of induction. }
	We will first show the base case of $i=2$ holds true for the induction. 
	From Eq.~\eqref{eq:bellman-gm1}-\eqref{eq:bellman-gm2}, 
	we get 
	\begin{equation*}
	m(b_1) = m(b_0) - a_{1,0}\Delta(0).
	\end{equation*}
	Since $a_{i,j}$ are the sum of probabilities, they are nonnegative. 
	Therefore, $-a_{1,0}\Delta(0) \ge 0$ from the hypothesis, 
	and the above equation implies that $m(b_1) \ge m(b_0)$. 
	Since $m$ is a nonincreasing function, it follows that $b_1 \le b_0$. 
	Further, from the Lipschitz-1 continuity of $m$, we get $m(b_1)-m(b_0) \le \abs{b_1-b_0} = b_0 - b_1$. 
	It follows that 
	$b_1 \le b_0 + a_{1,0}\Delta(0)$.  
	From the definition of sequence $b$ in Eq.~\eqref{eqn:ArgMax} and the definition of $a_{ij} = \sum_{l=0}^{j}\alpha_{i,i-l}$, 
	we get 
	\begin{equation*}
	\alpha_{2,0}\Delta(1) 
	\le (\alpha_{1,0} + \alpha_{2,2})\Delta(0).
	\end{equation*}
	Since $0 < \alpha_{2,0} < \alpha_{1,0} + \alpha_{2,2}$, we see that $\Delta(1) < \Delta(0) \le 0$. 
	Using the fact $\Delta(1) < \Delta(0) \le 0$ and from the definition of $a_{ij}=\sum_{l=0}^j\alpha_{i,i-l}$, 
	we can write the following inequality 
	\begin{align*}
	a_{2,0}\Delta(0) + a_{2,1}\Delta(1) &< (a_{2,0}+a_{2,1})\Delta(0) \\ 
	&= (2\alpha_{2,2}+\alpha_{2,1})\Delta(0)  \\
	&= 2\alpha_{1,1}\Delta(0) \le a_{1,0}\Delta(0).
	\end{align*}
	
	\textbf{Step 2: Inductive step of induction. } 
	We have shown the base case of $i=2$ holds true for the induction, 
	and assume the inductive hypothesis in Eqs.~\eqref{eqn:InductiveHypothesis1},~\eqref{eqn:InductiveHypothesis1} holds true for some $i \ge 2$. 
	From Eq.~\eqref{eq:bellman-gm2}, we can write the difference for $i \in [K-1]$
	\EQ{
		m(b_i)-m(b_{i-1}) = \sum_{j=0}^{i-2}a_{i-1,j}\Delta(j) - \sum_{j=0}^{i-1}a_{i,j}\Delta(j). 
	} 
	From the inductive hypothesis for $i$, it follows that $m(b_i) > m(b_{i-1})$. 
	Following the similar discussion to the base case, 
	from the monotone nonincreasing and Lipschitz-1 continuity of the map $m$, 
	it follows that $b_i < b_{i-1} + m(b_{i-1})-m(b_i)$.   
	Using Eq.~\eqref{eq:bellman-gm2} to write the difference $m(b_i)-m(b_{i-1})$, 
	sequence $b = (b_i: i \in \cX^\prime)$ defined in Eq.~\eqref{eqn:ArgMax} to write the difference $b_{i-1}-b_i$,
	and substituting $a_{i,j} = \sum_{l=0}^j\alpha_{i,i-l}$ 
	where $\alpha_{k,j}$ is the probability of $j$ departures from state $k$ in Eq.~\eqref{eqn:NumDeparturesSampledProcess}, 
	we get  
	\begin{align*}{
		\alpha_{i+1,0}\Delta(i) \le &\sum_{j=0}^{i-2}(a_{i+1,j}-a_{i,j})\Delta(j)\\
		&~~+ (\alpha_{i,0}+a_{i+1,i-1})\Delta(i-1).
	}\end{align*}
	From Lemma~\ref{lem:ProbDepartProps}(a), we have $a_{i+1,j}\le a_{i,j}$ for all $j \le i$ and from inductive hypothesis $\Delta(i-1) = \min_{j\le i-1}\Delta(j) \le 0$.  
	Further, from Lemma~\ref{lem:ProbDepartProps}(c), we have $\sum_{j=1}^{i}a_{i+1,j}-\sum_{j=1}^{i-1}a_{i,j} = \alpha_{1,1}$. 
	Using these three facts in the above equation, we get
	\EQ{
		\alpha_{i+1,0}\Delta(i) \le (\alpha_{i,0}-a_{i+1,i}+a_{i,i-1}+\alpha_{1,1})\Delta(i-1). 
	}
	From Lemma~\ref{lem:ProbDepartProps}(b), we have $a_{i,i} = 1 -\alpha_{i,0}$ for all $i \in \cX^\prime$ 
	and hence we get $\alpha_{i+1,0}\Delta(i) \le (\alpha_{i+1,0}+\alpha_{1,1})\Delta(i-1)$. 
	Since $\alpha_{i,j}$ are probabilities, the first inductive result $\Delta(i) < \Delta(i-1)$ follows.  
	Together with the inductive hypothesis in Eqs.~\eqref{eqn:InductiveHypothesis1},~\eqref{eqn:InductiveHypothesis1}, 
	we get $\min_{j\le i} \Delta(j) = \Delta(i) \le 0$. 
	From Lemma~\ref{lem:ProbDepartProps}(a), we have $a_{i+1,j} \le a_{i,j}$ for all $i \in \cX^\prime$ and $j\le i$. 
	Therefore, $\sum_{j=0}^{i-1}(a_{i+1,j}-a_{i,j})\Delta(j) \le \sum_{j=0}^{i-1}(a_{i+1,j}-a_{i,j})\Delta(i)$, 
	and we can write the difference 
	\begin{align*}
	\sum_{j=0}^{i}a_{i+1,j}\Delta(j) - \sum_{j=0}^{i-1}a_{i,j}&\Delta(j) \\ 
	&\le (\sum_{j-0}^ia_{i+1,j}-\sum_{j=0}^{i-1}a_{i,j})\Delta(i). 
	\end{align*} 
	From Lemma~\ref{lem:ProbDepartProps}(c), 
	we have $\sum_{j=0}^ia_{i+1,j}-\sum_{j=0}^{i-1}a_{i,j} = \alpha_{1,1}$ for all $i\in\cX^\prime$. 
	Substituting this result in the above equation, we get the second inductive result, 
	and this completes the induction. 
	\item Recall that the optimal price in state~$i$ is $u_i^\ast = \arg\max_u f(b_i,u)$ for the map $f$ defined in Eq.~\eqref{eqn:def-m}, 
	where the sequence $b = (b_i: i \in \cX^\prime)$ is positive and increasing from part (b). 
	Therefore, it follows from Lemma~\ref{lem-m-monotonicity}(c) that the optimal price $u_i^\ast$ is increasing with the  number of busy servers $i$.  
\end{enumerate}

\end{document}